\documentclass[a4paper,10pt]{article}
\usepackage{psfrag}
\usepackage{fancyhdr}
\usepackage{mathrsfs}
\usepackage{verbatim}
\usepackage[ansinew]{inputenc}
\usepackage{amsmath,amsthm,amssymb,amscd}
\usepackage[all]{xy}
\usepackage{xspace}
\usepackage[dvips]{color}
\usepackage{epsfig}
\usepackage{makeidx}
\usepackage{pb-diagram}
\usepackage{amsfonts}
\usepackage{graphicx}
\usepackage{subfigure}
\usepackage[absolute]{textpos}

\newcommand{\T}{\mathbb{T}}

\newcommand{\N}{\mathbb{N}}
\newcommand{\Z}{\mathbb{Z}}

\newcommand{\R}{\mathbb{R}}

\newcommand{\D}{\mathbb{D}}

\newcommand{\cont}{\mathcal{C}}

\DeclareMathOperator{\homeo}{Homeo(S)}
\DeclareMathOperator{\inte}{Int}

\DeclareMathOperator{\lleno}{Fill}

\newtheorem*{teo1}{Theorem A}
\newtheorem*{teo2}{Theorem B}
\newtheorem{teo}{Theorem}[section]
\newtheorem{cor}[teo]{Corollary}
\newtheorem{lema}[teo]{Lemma}
\newtheorem{prop}[teo]{Proposition}
\newtheorem*{teo0}{Theorem}
\newtheorem*{cor0}{Corollary}
\theoremstyle{definition}

\theoremstyle{remark}
\makeindex

\author{Alejandro Passeggi and Juliana Xavier}

\begin{document}

\title{Minimal sets for surface homeomorphisms}

\date{}
\title{A classification of minimal sets for surface homeomorphisms.}
\maketitle

\begin{abstract}

We classify minimal sets of (closed and oriented) hyperbolic surface homeomorphisms by studying the connected components of their complement. This extends the classification given by F. Kwakkel, T. Jager and A. Passeggi in the torus (see \cite{Kwajapa}).

The given classification is studied in the non-wandering setting and in light of the Nielsen-Thurston theory.

\end{abstract}

\begin{section}{Introduction.}\label{intro}

In \cite{Ferry} F. Kwakkel classified minimal sets of non-resonant toral homeomorphisms, that is, homeomorphisms which are homotopic to the identity and whose rotation set is a totally irrational vector. The classification was done studying the connected components in the complement of a minimal set. Later, in \cite{Kwajapa} F. Kwakkel, T. J\"ager and A. Passeggi extended this classification to general toral homeomorphism. Further, the authors studied consequences of the classification in the non-wandering setting and in the isotopy class of Anosov diffeomorphisms.

In this article we generalize this kind of classification to closed and oriented hyperbolic surfaces. Interesting differences with the toral case occur due to global topology. We also study the problem under the non-wandering assumption, and according to the isotopy class of the homeomemorphism, in light of the Nielsen-Thurston classification.

The main results of the article are contained in Theorem A and Theorem B; the former gives an initial classification depending on the homotopy type of the connected components of the complement of the minimal set, and the latter is a description of the topology of the minimal set in the case where there is a connected component of its complement that is homotopically non-trivial but neither contains all the homotopy of the surface.

\ Let us introduce now the preliminary definitions. All surfaces in this paper are assumed to be oriented. For a given closed surface $S$ we denote by $g(S)$ its genus. We consider hyperbolic surfaces as the quotient $S=\D/G$ where $\D$ is the open unit disk equipped with the Poincar\'{e} metric and $G$ the group of covering transformations. The universal covering projection is denoted by $p:\D\to S$. For any set $K\subset S$ we denote by $\Pi_0(K)$ the set of connected components of $K$.

\ We say that an open and connected set $V\subset S$ is essential if $$0\neq i_*(\pi_1(V,x_0))\neq \pi_1 (S,x_0),$$\noindent where $i_*:\pi_1(V,x_0)\to \pi_1 (S,x_0) $ is the homomorphism of fundamental groups induced by the inclusion map. For a compact set $C\subset S$ we define the set of essential connected components of its complement by $\mathcal{E}(\mathcal{C}^c)$.

We say that a minimal set $\mathcal{M}\subset S$ of $f\in\homeo$ is an extension of a minimal Cantor set (of a periodic orbit) if there exists a map $h:S\rightarrow S$ homotopic to the identity such that:

\begin{itemize}

\item[(i)] $h\circ f=\hat{f}\circ h$ for some $\hat{f}\in\homeo$;

\item[(ii)] $\mathcal{M}'=h(\mathcal{M})$ is a minimal cantor set (is a periodic orbit) of $\hat{f}$, and $h^{-1}(y)$ contains a unique connected component of $\mathcal{M}$ for every $y\in\mathcal{M}'$.

\end{itemize}

For a minimal set $\mathcal{M}\subset S$ of $f\in\homeo$, we say that an element $U\in\Pi_{0}(\mathcal{M}^c)$ is periodic if for some $n\in\N$ we have $f^n(U)=U$. Otherwise we say that $U$ is wandering. We say that $U\in\Pi_{0}(\mathcal{M}^c)$ is {\it bounded} if the connected components of $p^{-1}(U)$ are relatively compact.\\

The classification theorem is the following:

\begin{teo1}

Let $S$ be a closed hyperbolic surface and $\mathcal{M}\subset S$ minimal set of $f\in\homeo$. Then either:

\begin{itemize}

\item[(1)] the elements in $\Pi_0(\mathcal{M}^c)$ are all disks, with at least one unbounded. Moreover, any bounded disk in $\Pi_0(\mathcal{M}^c)$ is wandering;

\item[(2)] $\mathcal{M}=\cont_0\cup...\cup\cont_{n-1}$, where $\{\cont_i\}_{i=0}^{n-1}$ is a family of pairwise disjoint continua each containing a connected component of $\partial V$ with $V\in \mathcal{E}(\mathcal{M}^c)$, and $f(\cont_i)=\cont_{i+1}$ $\textrm{mod}(n)$ for all $i\in \{0, \ldots, n-1\}$.

\item[(3)] $\mathcal{M}$ is either an extension of a periodic orbit or an extension of a minimal cantor set.

\end{itemize}

\end{teo1}

We refer to a minimal set $\mathcal{M}$ as a type 1, 2 or 3 minimal set, whenever it is in the case 1, 2 or 3 of the classification theorem. Mental pictures to have in mind are:

\begin{itemize}
\item Type 3 minimal set: a fixed point;
\item  Type 2 minimal set: a non trivial simple loop ${\cal C}$ in $S$ which supports an irrational rotation (more interesting examples are described in Section \ref{examples});
\item Type 1 minimal set : consider a full geodesic lamination $\Lambda$ in $S$ of genus $g$ such that $S\backslash \Lambda$ is the union of two disks that lift to the universal covering space to ideal $2g$- gones with alternating orientations on its sides (see Figure \ref{laminacion}). This condition allows to complete the lamination to a flow in $S$ with a fixed point of index $1-g$ in each disk.  Moreover, the lamination is a minimal set for the flow, and one can show that there exists $t\in \R$ such that $\Lambda$ is minimal for the time $t$ of this flow (\cite{timet}).
\end{itemize}

\begin{figure}[ht]

\psfrag{U}{$U$}

\centerline{\includegraphics[scale=0.45]{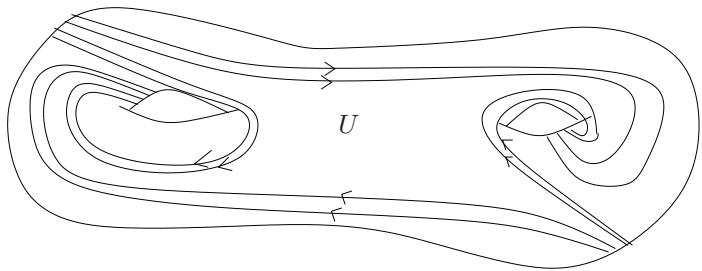}}

\caption{Type 1 minimal set}\label{laminacion}
\end{figure}

\newpage

A continuum $\cont\subset S$ is an essential \textit{circloid} if it verifies:

\begin{itemize}

\item[(i)] there exists some essential annuli $U$ such that $\cont\subset U$, and $U\setminus \cont$ has two unbounded components (components with non trivial fundamental group);

\item[(ii)] The continuum $\cont$ is minimal with respect the condition (i), that is, it does not strictly contain any other continuum with the property (i).

\end{itemize}

If $S$ is any closed surface (not necessarily hyperbolic), we say that a minimal set is of type 1 if $\Pi_0(\mathcal{M}^c)$ is a union of disks. This is consistent with the classification given in \cite{Kwajapa} when $S$ is the torus, and is the situation considered in \cite{MaNa} when $S$ is the sphere.

Given a continuum and type 2 minimal set $\cont$ of $f\in\homeo$, we say that it is reducible to a type 1 minimal set if there exists a closed surface $S'$ such that $g(S')<g(S)$, a homeomorphism $\hat f: S'\to S'$ and ${\cal C}_k'$ a type 1 minimal set for $\hat f$ such that $f|_{{\cal C}_k}$ is conjugated to $\hat f|_{{\cal C}_k'}$.

The second result which completes the classification is the following:

\begin{teo2}

Let $S$ be a closed hyperbolic surface and $\mathcal{M}=\cont_0\cup...\cup\cont_{n-1}$ a type 2 minimal set of $f\in\homeo$. Then, there exists $N>0$ such that $\cont_k$ is reducible to a type 1 minimal set for $f^N$. Furthermore,

\begin{itemize}

\item[(2i)] if $g(S') = 0$, then $\cont_k$ is the boundary of a circloid for every $k=0,...,n-1$ and $\mathcal{E}(\cont_k^c)\leq 2$. Further, every disk in $\Pi_0(\mathcal{M}^c)$ is wandering.

\item[(2ii)] if $g(S') > 0$ every bounded disk in $\mathcal{M}^c$ is wandering, and there are at least $\#\{{\cal E}({\cal C}_k^c)\}$  unbounded fixed disks in $\Pi_0 (S'\backslash{\cal C}_k')$.

\end{itemize}

\end{teo2}

Basically, this last result says that there are two ways for constructing type 2 minimal sets: one is to choose a finite family of pairwise disjoint compact essential annuli in $S$ where the minimal set will be supported; and the other one is to consider a type 1 minimal set $\cont'$ in some surface $S'$ and then extending the dynamics to a connected sum of $S'$ with some others surfaces. See the figures below, and the section of examples  for a detailed description.

\begin{figure}[h]\label{tipo2}

\centerline{\subfigure[(2i)]{\includegraphics[scale=0.40]{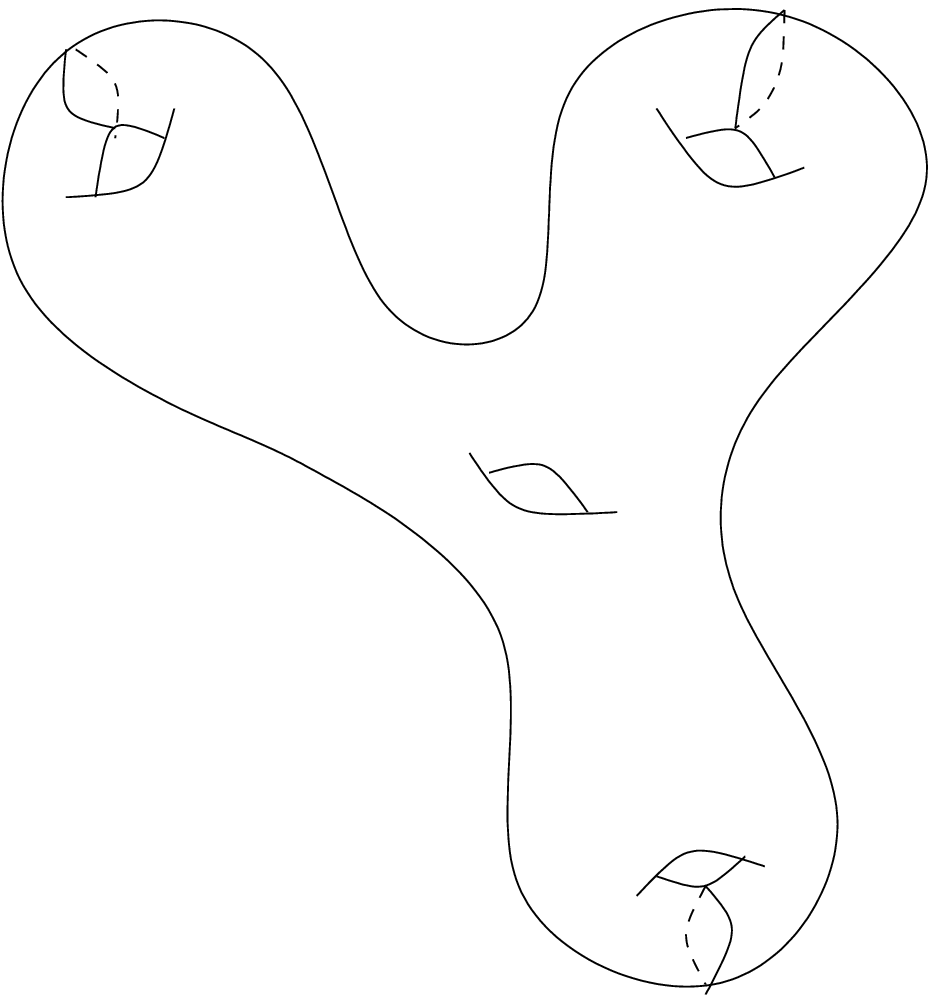}}\hspace{.45in}
\subfigure[(2ii)]{\includegraphics[scale=0.30]{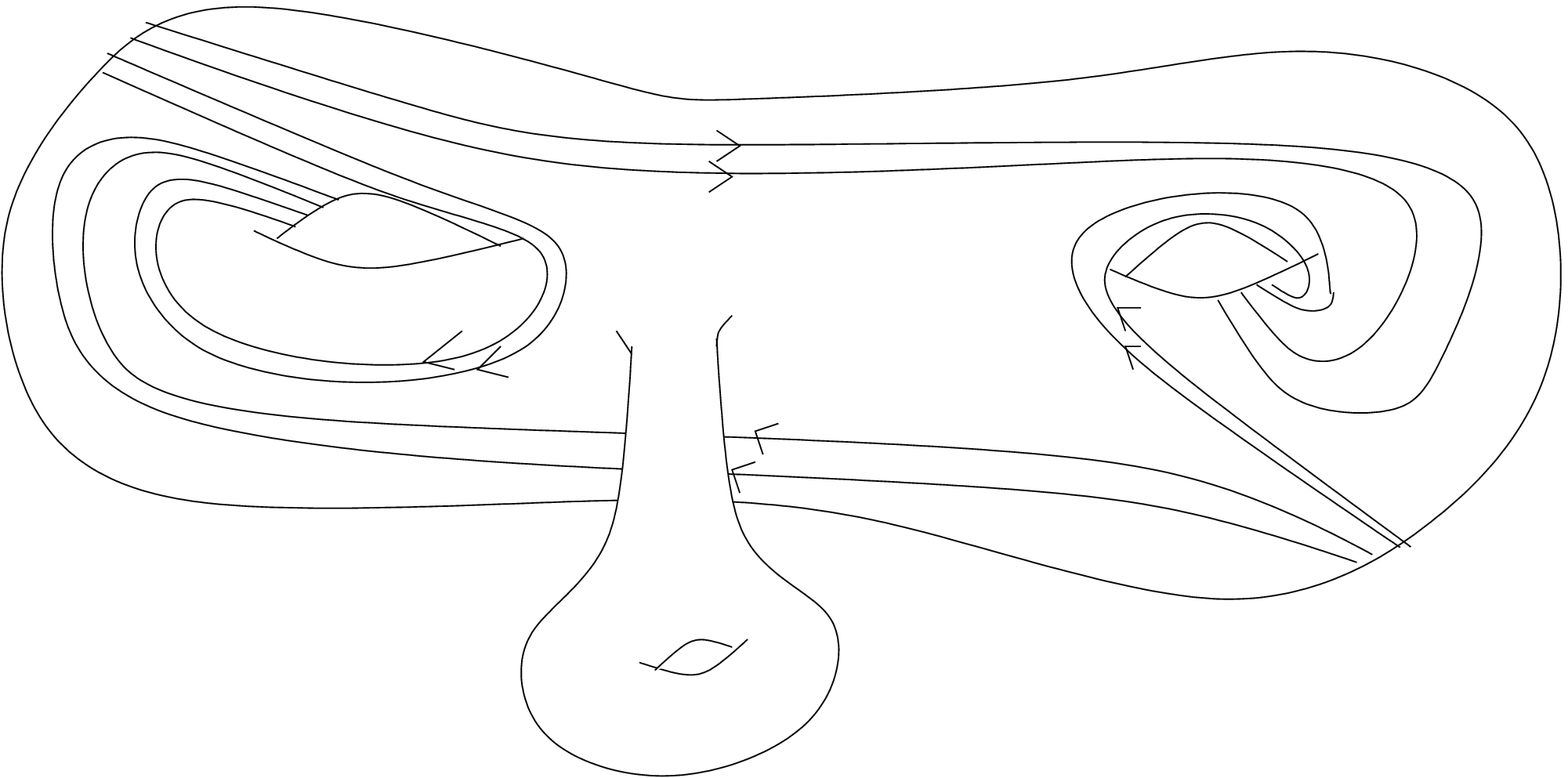}}}

\caption{Type 2 minimal sets }
\end{figure}

 \newpage More information about the conjugacy map between $f^N|_{{\cal C}_k}$ and $\hat f|_{{\cal C}_k'}$ can be obtained along the section \ref{type2}. In particular, information about the homotopy type of this map. We omit this here in sake of clarity.

The proofs of Theorem A and Theorem B are given in section \ref{proofAB}. After proving these theorems, we consider the non-wandering case. Recall that $f\in\homeo$ is non-wandering if for any open set $U\subset S$ there exists $n\in\N$ such that $f^n(U)\cap U\neq \emptyset$. In Section \ref{nonwand}, based in a result by A. Koropeki (\cite{koro}), we prove the following:

\begin{teo0}

Let $S$ be a closed hyperbolic surface, and $\mathcal{M}\subset S$ a minimal set for $f\in\textrm{Homeo}_{nw}(S)$. Then either:

\begin{itemize}

\item[$(1^{nw})$] $\mathcal{M}=\cont_0\cup...\cup\cont_{n-1}$ where $\{\cont_i\}_{i=0}^{n-1}$ is a family of pairwise disjoint circloids each containing a connected component of $\partial V$ with $V\in \mathcal{E}(\mathcal{M}^c)$, and $f(\cont_i)=\cont_{i+1}$ $\textrm{mod}(n)$ for $i=0...n-1$. Further there are no disks in $\Pi_0(\mathcal{M}^c)$;

\item[$(2^{nw})$] $\mathcal{M}$ is an extension of either a periodic orbit or a minimal cantor set.

\end{itemize}

\end{teo0}

Note the slight difference between the cases $(1^{nw})$ of last theorem and $(2i)$ of Theorem B. In the non-wandering case we have that the continua are circloids, while in the general case we only know that they are the boundary of a circloid. Finally, we devote Section \ref{nt} to the study of the classification depending on the isotopy class of $f$. In light of Nielsen-Thurston theory, we define the following set:

$$\mathcal{D}=\{f\in\homeo|\ f\sim g\ :\ g \mbox{ is either pA or reducible with pA components}\},$$

where ``$\sim$'' is the homotopy relation, and ``pA'' refers to pseudo-Anosov homeomorphisms. We obtain the following:

\begin{teo0}

Let $S$ be a closed hyperbolic surface. Then type one minimal sets can only occur for homeomorphisms in $\mathcal{D}^c$.

\end{teo0}

In other words, type 1 minimal sets can only occur for homeomorphisms which are either periodic or reducible without pseudo-Anosov components. We only know examples of type 1 minimal sets for periodic homeomorphisms (indeed only for homotopic to the identity maps), and we think that only for this class these minimal sets can exists. So, we pose the following question:

\textit{Do type 1 minimal sets exist for homeomorphisms in the homotopy class of reducible maps without pA components?
}

As a direct consequence of the last theorem we prove the following:

\begin{cor0}

Let $S$ be a closed hyperbolic surface and $f\in\homeo$ homotopic to a pseudo-Anosov homeomorphism. Then every minimal set of $\mathcal{M}$ is type 3.

\end{cor0}

\end{section}

\begin{section}{A preliminary classification.}\label{homotopyofcomp}

Let $S$ be a compact hyperbolic surface, $\tilde S$ its universal covering space, and $p:\tilde S \to S$ the covering projection. We identify  $\tilde S$ with the Poincar\`e disk $\D$, and $S^1=\partial \D$ with $\partial \tilde S$.

Let $U\subset S$ be open and connected, $i:U\to S$ the inclusion map and $i_{*}:\pi_1(U,x_0)\rightarrow \pi_1(S,x_0)$ the induced homomorphism of fundamental groups. We say that $U$ is {\it trivial} if $\textrm{Im}(i_{*})=0$;
we say that $U$ is ${\it essential}$ if $$0\neq \textrm{Im}(i_{*})\neq \pi_1(S,x_0); $$\noindent finally, we say that $U$ is {\it fully essential} if $i_{*}$ is surjective.

If $U\subset S$ is connected, we say that $U$ is {\it bounded} if the connected components of $p^{-1}(U)$ are relatively compact.

Let $f:S\to S$ be a homeomorphism and ${\cal M}$ a minimal set for $f$.  We note $\Pi_0 ({\cal M}^c)$ for the set
of connected components of $S\backslash{\cal M}$.

In this section we make the following preliminary classification of minimal sets:

\begin{teo}\label{prel}  One and only one of the following holds:

\begin{enumerate}

\item Every $U\in\Pi_0(\mathcal{M}^c)$ is a disk;

\item There exists an essential component $U\in\Pi_0(\mathcal{M}^c)$, and
\\$i_{*}:\pi_1(V,x_0)\rightarrow \pi_1(S,x_0)$ is injective for all $V\in\Pi_0(\mathcal{M}^c)$;

\item There exists a fully essential component $V\in\Pi_0(\mathcal{M}^c)$ and every other $U\in\Pi_0(\mathcal{M}^c)$ is a bounded disk.

\end{enumerate}

\end{teo}

\noindent Figure \ref{laminacion} illustrates the first item, and Figure \ref{tipo2} the second one.\\

We begin with merely topological facts.

\begin{lema}\label{ltop}  If $V$ is a fully essential open and connected subset of $S$, $U\subset S$ is a connected set, and $V\cap U = \emptyset$ then $U$ is bounded.
\end{lema}

\begin{proof} If $U$ is not bounded, and $\tilde U$ is a connected component of $p^{-1}(U)$, there exists $x\in S^1\cap \overline{\tilde U}$.  Take $x_0\in \tilde U$ and let $\alpha$ be the geodesic joining $x_0$ and $x$, where $x_0$ is considered in the interior of a fundamental geodesic polygon $P$. As $\alpha$ is unbounded,  there exists a geodesic $\gamma$ containing one of the sides of $P$ and an infinite sequence $(g_n)_{n\in N}$ of covering transformations such that $\alpha \cap g_n(\tilde \gamma)\neq \emptyset$. Let $W_n$ be the connected component of $\D\backslash g_n(\gamma)$ containing $x$ in its closure.  Then,  $(W_n)_{n\in \N}$ is a basis of neighborhoods of $x$. Furthermore, $\overline{g_n(\gamma)}$ separates $x_0$ from $x$ for all $n\in \N$. As $V$ is fully essential, one can find $\beta\subset V$ homotopic to $\gamma$.  Furthermore, for any $n\in\N$, there exists a lift $\beta _n$ of $\beta$ asymptotic to $g_n(\gamma)$. So, there exists $\tilde \beta$ which lifts $\beta$ such that $\overline{\tilde \beta}$ separates $x$ from $x_0\in \overline \D$.  This implies that $\tilde \beta\cap\tilde U\neq \emptyset$ which contradicts the fact $V\cap U = \emptyset$.
\end{proof}

%\begin{lema}\label{topo}  If $V$ is a fully essential open and connected subset of $S$, and $V\cap U = \emptyset$ then $U$ is not essential.
%\end{lema}
%
%\begin{proof} Suppose $U$ is essential and $V\cap U = \emptyset$. Then, if $\tilde U$ is a lift of $U$, then there exists two different points $x,y \in S^1$ such that $\{x,y\}\subset \overline{\tilde U}$. As $V$ is fully essential, %there exists a lift $\tilde V$ of $V$ and  two different points $z,t \in S^1 \cap \overline{\tilde V}$ such that $\{x,y\}$ separates $z$ and $t$ in $S^1$.  This is a contradiction, since $\tilde U \cap \tilde V \neq \emptyset$ %implies $U\cap V\neq\emptyset$.
%\end{proof}

\begin{lema}\label{ltop2} If $X\subset S$ is contained in a finite union of pairwise disjoint closed
disks $D_1, \ldots, D_n$, then $i_* : \pi_1(S\backslash X, x_0) \to \pi_1(S, x_0)$ is surjective for any $x_0\in S\backslash \cup_{i=1}^n D_i$.
\end{lema}

\begin{proof}Let $X \subset \cup_{i=1}^n D_i$, with the $D_i$'s closed and pairwise disjoint disks. Then,
$p ^{-1}(\cup_{i=1}^n D_i)$ is a union of pairwise disjoint closed disks containing $p^{-1}(X)$.
Indeed, any two lifts $\tilde D_i$ and $\tilde D_j$ of $D_i$ and $D_j$ with $i\neq j$ cannot intersect
as $D_i \cap D_j =\emptyset$ . Moreover, for any non trivial $T \in G$ and $\tilde D_i$ a lift of $D_i$,
$\tilde D_i\cap T(\tilde D_i)= \emptyset$ because $D_i$ is a disk and cannot contain a non trivial loop. In
particular, $\tilde S \backslash p^{-1}(\cup_{i=1}^n D_i)$ is arc connected. If $i_* : \pi_1(S\backslash X, x_0) \to \pi_1(S, x_0)$ is
not surjective, there exists $T\in G$ such that for any lift $\tilde x$ of  $x_0$, $\tilde x$ and $T(\tilde x)$ cannot
be joined by a path in $\tilde S\backslash p^{-1}(X)$. As $\tilde S \backslash p^{-1}(\cup_{i=1}^n D_i)$ is arc connected and both $\tilde x$ and $T(\tilde x)$ belong to $\tilde S \backslash p^{-1}(\cup_{i=1}^n D_i)$, this is a contradiction.

\end{proof}

We will show that if there is no fully essential component in $\Pi_0 ({\cal M}^c)$, then for any $U\in \Pi_0 ({\cal M}^c)$, the map $i_{*}:\pi_1(U,x_0)\rightarrow \pi_1(S,x_0)$ is injective.

\begin{lema} If there exists $U\in \Pi_0 ({\cal M}^c)$ such that the map
$i_{*}: \pi_1(U,x_0)\to \pi_1(S, x_0)$ is not injective, then there exists a finite number of pairwise disjoint closed disks in $S$, $D_1, \ldots, D_n$ such that
${\cal M}\subset \inte (\cup _{i=1}^n D_i)$.

\end{lema}

\begin{proof} Take $U\in \Pi_0 ({\cal M}^c)$ such that the map
$i_{*}: \pi_1(U, x_0)\to \pi_1(S, x_0)$ is not injective. Then, there exists a simple closed curve $\gamma \subset U$ such that $\gamma$ bounds an open disk
$D\subset S$, but $D\cap {\cal M}\neq \emptyset$.

By minimality, for any ${\cal C}\in \Pi_0 ({\cal M})$ there exists $n\geq 0$ such that $f^n({\cal C})\subset D$. Indeed, $f^n({\cal C})\cap D\neq \emptyset$ for some
$n\geq 0$, and $f^n({\cal C})\cap \partial D = \emptyset$ as $\partial D = \gamma \subset U \subset {\cal M}^c$, and $f^n({\cal C})$ is connected. For ${\cal C}\in \Pi_0 ({\cal M})$, let
$k_{\cal C} = \min \{n\geq 0: f^n(\cal C)\subset D\}$.  Then, for all ${\cal C}\in \Pi_0 ({\cal M})$  there exists an open disk $D_{\cal C}$ such that
$k_{\cal C} = \min \{n\geq 0: f^n(D_{\cal C})\cap D \neq \emptyset\}$, and $f^{k_{\cal C}}(D_{\cal C})\subset D$.
In particular, for any
${\cal C}'\subset D_{\cal C}$ $k_{\cal C '} = k_{\cal C } $.  Evidently, ${\cal M}\subset \cup_{{\cal C} \in  \Pi_0 ({\cal M})} D_{\cal C}$ and we can take a finite
subcover ${\cal M}\subset \cup_{i=1}^n D_{{\cal C} _i}$.  For all $i=1, \ldots, n$, let  $k_i = k_{{\cal C} _i}$ and $D_i = D_{{\cal C}_i}$.  It is clear that
$k_i\neq k_j$ implies $D_i\cap D_j = \emptyset$.  Besides, we can suppose that if $k_i = k_j$, then $D_i=D_j$. Indeed, if
$k_i = k_j=k$, one can find a disk $D'\subset D$ such that $f^k(D_i\cup D_j)\subset D'$.  Then, $D_k = f^{-k}(D')$ is a disk
containing $D_i\cup D_j$.  Replacing $D_k $ by a smaller disk if necessary, one has that $k=\min \{n\geq 0 : f^n(D_k)\cap D\neq\emptyset\}$.  So, if  $k_i = k_j$, and $D_i\neq D_j$ , we replace $D_i$ and $D_j$ by $D_k$.

As ${\cal M}$ is closed, we can suppose by taking
smaller disks if necessary that the sets $\overline D_i$, $i=1, \ldots, n$ are pairwise disjoint, and we are done.

\end{proof}

\begin{cor}\label{iny} Suppose that there exists $U\in \Pi_0 ({\cal M}^c)$ such that the map
$i_{*}: \pi_1(U,x_0)\to \pi_1(S, x_0)$ is not injective. Then there exists a fully essential component $V\in \Pi_0({\cal M}^c)$.

\end{cor}

\begin{proof} We know that ${\cal M}\subset \inte(\cup _{i=1}^n D_i)$, where the $D_i$'s are pairwise disjoint closed disks in $S$.  Then,
$V' = S\backslash \cup _{i=1}^n D_i$ is open and connected, $V'\subset {\cal M}^c$, and $i_{*}: \pi_1(V',x_0)\to \pi_1(S, x_0)$ is surjective (see Lemma \ref{ltop2}).  We
take $V$ to be the connected component of
${\cal M}^c$ containing $V'$.

\end{proof}

\begin{lema}\label{bounded} If $U\in \Pi_0 ({\cal M}^c)$ is bounded, then $U$ is a topological disk.
\end{lema}

\begin{proof} As $U$ is bounded, if $\tilde U$ is a connected component of $p^{-1}(U)$, then $\partial \tilde U$ has a unique connected component ${\cal C}$ that
separates $S^1$ from $\tilde U$. Let $D$ be the bounded connected component of $\D\backslash {\cal C}$.  We claim that for any lift $\tilde f$ of $f$ and any $k\in\Z$, either $\tilde f^k(D)=D$, or $\tilde f^k(D) \cap D= \emptyset$.  It is clear that one can not have $\tilde f^k(D)\subsetneq D$ or $D\subsetneq \tilde f^k(D)$ as this would imply the existence of a wandering point in ${\cal M}$.  As $D$ and $f^k(D)$ are toplogical disks,  $\tilde f^k(D) \cap D\neq  \emptyset$ but $\tilde f^k(D)\neq D$ is impossible because this implies that ${\cal C}$ cannot be a connected component of $\partial \tilde U$, proving the claim. If $U$ is not a disk, then ${\cal M}\cap p(D)\neq \emptyset$, and there
exists a lift $\tilde f$ of $f$ and $k\in \Z$ such that $\tilde f^k(D) \cap D\neq \emptyset$.  Therefore, $\tilde f^k(D)=D$, which implies $\tilde f^k({\cal C}) = {\cal C}$.  If ${\cal K}$ is the connected component of ${\cal M}$ containing $p({\cal C})$, then the (periodic) orbit of ${\cal K}$ is a closed an invariant set, which is different from ${\cal M}$, as ${\cal M}\cap p(D)\neq \emptyset$.  This contradiction finishes the proof.

\end{proof}

\begin{lema}\label{inj} If $U\in \Pi_0 ({\cal M}^c)$ is not fully essential, then
\\$i_{*}: \pi_1(U,x_0)\to \pi_1(S, x_0)$ is injective.
\end{lema}

\begin{proof} If $U$ is essential, then there are no fully essential components in $\Pi_0 ({\cal M}^c)$ (see Lemma \ref{ltop}).  So, $i_{*}: \pi_1(U,x_0)\to \pi_1(S, x_0)$ is injective, this being the content of Corollary \ref{iny}.
If $U$ is trivial and there are no fully essential components in $\Pi_0 ({\cal M}^c)$, then  $i_{*}: \pi_1(U,x_0)\to \pi_1(S, x_0)$ is injective again by Corollary \ref{iny}.  If there is a fully essential component $V\in \Pi_0 ({\cal M}^c)$, then $U$ is bounded (see Lemma \ref{ltop}) and so by Lemma \ref{bounded} $i_{*}: \pi_1(U,x_0)\to \pi_1(S, x_0)$ is injective.
\end{proof}

The proof of Theorem \ref{prel} follows.

\begin{proof} If (1) holds, then evidently neither (2) nor (3) hold.  If (2) holds, then evidently (1) does not hold, and (3) cannot hold due to lemma \ref{ltop}. If (3) holds, then evidently (1) does not hold, and (2) cannot hold because of lemma \ref{ltop}. So, only one of them can hold.

Suppose that (3) does not hold.  So, by Corollary \ref{iny}, for every $U\in \Pi_0 ({\cal M}^c)$ the map
$i_{*}: \pi_1(U,x_0)\to \pi_1(S, x_0)$ is injective.  Then necessarily  (1) or (2) hold and we are done.
\end{proof}

\begin{cor}\label{puncturedsurf} If $U\in \Pi_0({\cal M}^c)$ is essential, then $U$ is homeomorphic to a finitely punctured closed surface.
\end{cor}

\begin{proof}\label{ri} By a theorem of Richards (see \cite{richards} Theorem 3), $U$ is homeomorphic to $\Sigma \backslash X$, where $\Sigma$ is a closed surface and $X$ is a closed and totally disconnected set.  We just have to prove that $X$ has finitely many connected components.  If this was not the case, there would be infinitely many homotopically non trivial, pairwise non freely homotopic and pairwise disjoint loops $(\gamma _n)_{n\in \N}\subset U$.  As $i_{*}: \pi_1(U,x_0)\to \pi_1(S, x_0)$ is injective and $S$ is a closed surface, this is a contradiction.
\end{proof}

\begin{cor}\label{finiteboundcomp} If $U\in \Pi_0({\cal M}^c)$ is essential, then $\partial U$ has finitely many connected components.
\end{cor}

\begin{proof}  By the previous Corollary, $U$ is homeomorphic to a finitely punctured closed surface. It follows that the number of punctures is greater or equal than $\# \Pi_0(\partial U)$.
\end{proof}

\begin{section}{Periodic disks in $\Pi_0 ({\cal M}^c)$}

We will say that a minimal set ${\cal M}$ is of type 1, 2 or 3 if it corresponds to cases  (1), (2) or (3) of Theorem \ref{prel}.

We devote this section to proving that there cannot be bounded periodic disks in $\Pi_0(\mathcal{M}^c)$ for a minimal set of type (1) or (2).  As an outcome of this fact we obtain that for type (1) minimal sets there is always a periodic unbounded disk in $\Pi_0(\mathcal{M}^c)$ .

We remark that the following four lemmas are just an adaptation of the ideas in [JKP] to higher genus.

\begin{lema}\label{per} If ${\cal M}$ is connected and $U\in\Pi_0(\mathcal{M}^c)$ is a periodic disk of period $p$, then ${\cal M} = \partial U = \ldots =\partial f^{p-1} (U)$.
\end{lema}

\begin{proof}  The set $\partial U \cup \ldots \cup \partial f^{p-1} (U)$ is closed, invariant and contained in ${\cal M}$.  So, ${\cal M} = \partial U \cup \ldots \cup \partial f^{p-1} (U)$.  Let
$r(x) = \# \{0\leq k <p: x\in \partial f^k(U)\}$. Then, for any $k_0\in \{1, \ldots, p\}$ the set $r^{-1}(\{k: k\geq k_0\})$ is closed and invariant, and therefore empty or equal to ${\cal M}$.  This implies that there exists $m\geq 1$ such that $r(x) = m$ for all $x\in{\cal M}$.

For every $x\in {\cal M}$ define $I_x = \cap _{i=1}^m \partial f^{k_i}(U)$, where $k_1, \ldots ,k_m\in\{0,...,p-1\}$ are such that $x\in \partial f^{k_i} (U)$. Then, $I_x$ is closed and
${\cal M} = I_{x_1}\cup \ldots \cup I_{x_n}$ with $I_{x_i}\neq I_{x_j}$ if $i\neq j$.  Furthermore, $I_{x_i}\cap I_{x_j} = \emptyset$ if $i\neq j$, because otherwise there would exist $z\in {\cal M}$ with $r(z)>m$.  As ${\cal M}$ is connected, we have ${\cal M} = I_{x_1} = \ldots = I_{x_n}$, which implies ${\cal M} = \partial U = \ldots =\partial f^{p-1} (U)$.
\end{proof}

\begin{lema}  Suppose that ${\cal M}$ is connected and there exists $U\in\Pi_0(\mathcal{M}^c)$ such that $U$ is bounded and periodic. Let $\tilde U$ be a connected component of $p^{-1}(U)$.  Then, for each $x\in {\cal M} = \partial U$ and $\tilde x \in \partial \tilde U$, $r(x) = \# \{T\in G: T(\tilde x)\in\partial \tilde U\}$ is a well defined constant function in ${\cal M}$.
\end{lema}

\begin{proof} As $U$ is bounded, for each $x\in {\cal M} = \partial U$ and $\tilde x\in\partial \tilde U$ the set
$$\{T\in G: T(\tilde x)\in\partial\tilde U\}$$\noindent is finite.  Moreover, $$r(x)= \# \{T\in G: T(\tilde x)\in\partial \tilde U\}$$\noindent is independent of
 $\tilde x \in \partial \tilde U \cap p^{-1}(x)$.  So, we have a well defined function $r: {\cal M}\to \N$.  We claim that for any $n_0\in \N$, the set
 $r^{-1}(\{n<n_0\})$ is open.  Indeed, take $x\in r^{-1}(\{n<n_0\})$ and $\tilde x \in \partial \tilde U \cap p^{-1}(x)$.  Consider a neighborhood $V$ of
 $\tilde x$ such that the sets $T(V), T \in G,$ are pairwise disjoint and $p(V)$ is a neighborhood of $x$. There exists $n<n_0$ and $T_1, \ldots, T_n \in G$ such that for any $T\in G, T \notin \{T_1, \ldots, T_n\}$, $T(\tilde x)\notin \partial \tilde U$.  Taking $V$ smaller if necessary, we can suppose that $T(V)\cap \partial \tilde U = \emptyset$ for all $T\in G, T \notin \{T_1, \ldots, T_n\}$.  So, for any $z\in V$, $T(z)\in \partial \tilde U$ if and only if $T\in\{T_1, \ldots, T_n\}$.  In particular, $p(z)\in r^{-1}(\{n<n_0\})$.  So, $p(V)\subset  r^{-1}(\{n<n_0\})$ showing that $ r^{-1}(\{n<n_0\})$ is open.  Equivalently, the sets $X_{n_0} = r^{-1}\{n\geq n_0\}$ are closed for all $n_0\in \N$.  To see that these sets are also $f$- invariant it is enough to be able to lift $f$ to a homeomorphism $\tilde f:\tilde S\to \tilde S$ such that $\partial \tilde U$ is $\tilde f$-invariant. As ${\cal M} = \partial U = \partial f(U)$, if $\tilde g$ is any lift of $f$, then  $\tilde g (\partial \tilde U) = \partial \tilde g (U) = T(\partial \tilde U)$ for some $T\in G$. Then, $\tilde f = T^{-1}g$ is a lift of $f$ fixing $\partial \tilde U$. By minimality, one obtains that $r(x)$ is independent of $x\in{\cal M}$.
\end{proof}

\begin{lema}\label{T} Suppose that ${\cal M}$ is connected and that there exists $U\in\Pi_0(\mathcal{M}^c)$ such that $U$ is bounded and periodic.  Then ${\cal T} _x = \{T\in G: T(x)\in \partial \tilde U\}$ is independent of $x\in \partial \tilde U$.
\end{lema}

\begin{proof}The previous lemma implies that there exists $m\geq 0$ such that for any $x\in \partial \tilde U$, $\# {\cal T}_x =m$.  As $U$ is bounded, the family
 $\{{\cal T} _x : x\in \partial U\}$ is finite. For all $x\in \partial U$ define $C_x = \cap _{i=1}^m T_{x_i}(\partial U)$, where $x\in T_{x_i}(\partial U)$ for
all $i=1, \ldots, m$.  Then, each $C_x$ is closed and the family ${\cal F}=\{C_x:x\in \partial U\}$ is finite, ${\cal F}=\{ C_{x_1}, \ldots, C_{x_n}\}$.  Furthermore, $C_{x_i} \cap C_{x_j} = \emptyset$ if $i\neq j$: otherwise, there would exist $x\in \partial \tilde U$ with $\# {\cal T}_x > m$.  As $\partial \tilde U$ is connected, we have $\partial \tilde U = C_{x_1} = \ldots = C_{x_n}$ and we are done.

\end{proof}

\begin{lema}\label{con} If ${\cal M}$ is a connected type 1 or 2 minimal set and $U\in\Pi_0(\mathcal{M}^c)$ is bounded, then $U$ is wandering.
\end{lema}

\begin{proof} Otherwise, Lemma \ref{per} implies that ${\cal M} = \partial U$.  As $U$ is bounded, Lemma \ref{T} implies that for any lift $\tilde U$ of $U$ and
any $T\in G$ we have $T(\overline{\tilde U})\cap\overline{\tilde U} = \emptyset$. So, $\bigcup_{T\in G} T(\overline {\tilde U})$ can be covered by a pairwise disjoint family of closed disks, and there is a fully essential component in $\Pi_0 ({\cal M}^c)$
\end{proof}
\end{section}

\begin{lema}\label{35} If ${\cal M}$ is of type 1 or 2, then there is no bounded periodic disk in $\Pi_0(\mathcal{M}^c)$.
\end{lema}

\begin{proof} If $U\in\Pi_0(\mathcal{M}^c)$ is a bounded periodic disk, then ${\cal M} = \partial U \cup f(\partial U) \ldots \cup f^p(\partial U)$.  Let $C$ be
the connected component of ${\cal M}$ containing $\partial U$.  Then, $C$ is periodic and minimal for $f^q$.  So, by lemma \ref{con}, $C^c$ contains a
fully essential component. Therefore $f^i(C)$ contains a fully essential component for $i=1, \ldots, q$, and so does ${\cal M} = \cup_{i=1}^n f^i(C)$.  This
contradicts the fact that ${\cal M}$ is of type 1 or 2.
\end{proof}

We finish by proving:

\begin{lema}\label{36} If ${\cal M}$ is a type 1 minimal set, then there is at least one unbounded periodic disk in ${\cal M}^c$.

\end{lema}

\begin{proof} Any  homeomorphism of a surface of genus $g>1$ has a periodic point $p$ (\cite{fuller}). As bounded disks in ${\cal M}^c$ are wandering, $p\in {\cal M}$ or there exists one unbounded periodic disk in ${\cal M}^c$. If $p\in {\cal M}$, then ${\cal M}$ is the orbit of $p$, which contradicts that $\cal M$ is a type 1 minimal set.

\end{proof}

\end{section}

\begin{section}{Type 2 minimal sets.}\label{type2}

In this section we prove first the following:

\begin{prop}\label{cont} If  ${\cal M}$ is a type 2 minimal set, then ${\cal M} = {\cal C}_0\cup\ldots\cup{\cal C}_{n-1} $, where $({\cal C}_i)_{i=0,\ldots,n-1 }$ is a family of pairwise disjoint continua each containing a connected component of $\partial V$ for some $V\in\mathcal{E}(\mathcal{M}^c)$, and $f({\cal C})_i = {\cal C} _{i+1}\mod n$ for $i=0...n-1$.
\end{prop}

The idea of the proof is simple. First we prove that an essential component in $\Pi_0({\cal M}^c)$ must be periodic.  As such a component $U\in \Pi_0({\cal M}^c)$ exists because we are considering type 2 minimal sets, one has ${\cal M} = \partial U \cup\partial f(U)\cup \ldots \partial f^{p-1} (U)$, where $p$ is the period of $U$.  We finish by taking ${\cal C} _0$ as to be a connected component of ${\cal M}$ contained in $\partial U$.

After proving this, we construct for any continuum $\cont\subset S$ which is type two minimal set for some homeomorphism $f\in\homeo$, a surface $S^*$ containing $\cont$ with $\partial S^*$ non-empty family of essential loops and $n_0\in\N$, such that the following theorem holds:

\begin{teo}\label{typ2str2} There exists an homeomorphism $g$ homotopic to $f^{n_0}$ with $g(\beta) = \beta$ for all $\beta\in \partial S^*$, such that if $S'=S^*/_{\partial S^*}$ and $q: S^*\to S'$ is the quotient map, then:

\begin{enumerate}

\item\label{pri} $S'$ is a surface with $g(S')<g(S)$ and ${\cal C'}:={q({\cal C})}$ is a minimal set for the induced map $\hat{g}\in \textrm{Homeo}(S')$. Furthermore $q$ conjugates $f^{n_0}|_{\cont}$ with $\hat{g}|_{\cont'}$ for some $n_0\in\N$;

\item \label{segu} $S'\backslash {\cal C'}$ is a union of disks and there are at least $\# \mathcal{E}(S\backslash\cont^c)$ unbounded disks containing fixed points of $\hat g$.

\end{enumerate}

\end{teo}

These two results are crucial for the proofs of Theorem A and Theorem B.

\begin{subsection}{Essential components are periodic}

Throughout this section ${\cal M}$ will stand for a type 2 minimal set.

We begin by the following

\begin{lema}  If  $U\in \Pi_0({\cal M}^c)$ is essential and is not an annulus, then $U$ is periodic.
\end{lema}

\begin{proof}  Let $g$ be the genus of $U$.  If $g\neq 0$, then clearly $U$ is periodic, as $S$ is a compact surface. If $g=0$ as $i_{*}: \pi_1(U,x_0)\to \pi_1(S, x_0)$ is injective (see Lemma \ref{inj}) $U$ is homeomorphic to the $n$-punctured sphere, $n\geq 3$, as we are assuming that $U$ is not an annulus. Let $U_i = f^i(U), i \geq 0$.  For all $i\geq 0$, $U_i$ is an $n$- punctured sphere with $n\geq 3$ whose fundamental group injects in $\pi_1(S, x_0)$.  As $S$ is compact, necessarily $(U_i)_{i\geq 0}$ has to be a finite family.
\end{proof}

Let $\mathcal{A}$ be the family of all essential open annuli in $S$. To each element $U\in\mathcal{A}$ we can associate a subset $U_{\infty}$ of the circle at infinity  as follows. Take a simple loop $\alpha\subset U$ generating $\pi_1(U, x_U)$, and pick a lift $\tilde \alpha\subset \D$ of $\alpha$.  Let $a,b\in S^1$ such that $\overline {\tilde \alpha} \cap \partial \D = \{a,b\}$.  Let $U_{\infty} = \{T(a), T(b): T\in G\}$.  The set  $U_{\infty}$ is independent of the choice of $\alpha$ and $\tilde \alpha$.  So, for each $A\in {\cal A}$ there is a well defined subset $U_{\infty}\subset S^1$.

Notice that given two annuli $U,V\in {\cal A}$ we have $U_{\infty}=V_{\infty}$ if and only if there exists two essential loops $\alpha \subset U$, $\beta \subset V$, and two lifts $\tilde \alpha$ and $\tilde \beta$ of $\alpha $ and $\beta$ respectively, such that $\overline{\tilde{\alpha}}\cap S^1=\overline{\tilde{\beta}}\cap S^1$.

%simple generator $\gamma_U'$ of $\pi_1(U, x_U)$. Let $\gamma_U $ be the free homotopy class in $S$ of $ i_*(\gamma_U')\in \pi_1(S,x_U)$.  If $\beta$ is another generator of  $\pi_1(U, x_U)$, then the free homotopy class in $S$ of $i_*(\beta)$ belongs to $\{\gamma_U, \gamma_U^{-1}\}$.  So, for each $A\in {\cal A}$ there is a well defined element $\gamma _A \in \pi_1(S)/_{\gamma \sim \gamma^{-1}}$,  where we denote $\pi_1(S)$ the set of free homotopy classes in $S$.

\begin{lema} Let $U \in \mathcal{A}\cap \Pi_0({\cal M}^c)$, and define $${\cal F} = \{V\in \mathcal{A}\cap \Pi_0({\cal M}^c): U_{\infty} =V_{\infty}\}.$$  Then, there exists $A\in {\cal A}$ such that \begin{enumerate}
\item\label{uno} $A_{\infty} =U_{\infty}$;
\item \label{dos}for all $V\in{\cal F}, V\subset A$;
\item \label{tres}$\partial A \subset {\cal M}$;
\item \label{cuatro}for all $x\in \partial A$ and all neighborhood $N$ of $x$, there exists $V\in{\cal F}$ such that $\partial V\cap N\neq\emptyset$.

\end{enumerate}
\end{lema}

\begin{proof} Let $\gamma\subset U$ be a generator of $\pi_1(U, x_U)$, and $\tilde \gamma\subset \D$ a lift of $\gamma$.  Let $a,b\in S^1$ such that $\overline {\tilde \gamma}\cap S^1 = \{a,b\}$, and $\tilde U$ the connected component  of $p^{-1}(U)$ containing $\tilde \gamma$.

For all $V\in{\cal F}$, pick a loop $\alpha(V)\subset V$ generating $\pi_1(V,x_V)$ and let $\tilde {\alpha (V)}$ be the lift of $\alpha(V)$ such that $\overline{\tilde {\alpha (V)}}\cap S^1 = \{a,b\}$. Define $\tilde V$ to be the connected component of $p^{-1}(V)$ containing $\tilde {\alpha (V)}$.

Further lets $\mathcal{D}$ be the family of topological disks in $\D$ which are bounded by $\overline{\tilde{\alpha}}$ and $\overline{\tilde{\beta}}$, where $\tilde{\alpha}\subset \tilde{V}_0$ and $\tilde{\beta}\subset\tilde{V}_1$ lifts two essential loops $\alpha\subset V_1\in\mathcal{F}$ and $\beta\subset V_2\in\mathcal{F}$ respectively. Notice that in particular $\overline{\tilde{\alpha}}\cap S^1=\overline{\tilde{\beta}}\cap S^1=\{a,b\}$.

We define: 

$$\tilde{A}:=\bigcup_{D\in\mathcal{D}}D$$

The construction of $\tilde{A}$ implies for any compact set $K\subset \tilde{A}$, that $K\subset D$ for some $D\in\mathcal{D}$. Then we have that $\tilde{A}$ is open, connected and simply connected. 

We claim that $A:=p(\tilde{A})$ is an annulus. If $A$ is not an annulus there exists $T\in G$ with $T(\{a,b\})\neq\{a,b\}$ and $T(\tilde{A})\cap\tilde{A}\neq\emptyset$. This implies that we have $V_0,V_1\in\mathcal{F}$ with $T(\tilde{V}_1)\cap \tilde{V}_0\neq\emptyset$ and $\tilde{\xi}\subset T(\tilde{V}_1)$, which lifts an essential loop $\xi\subset V_1$, such that $\overline{\tilde{\xi}}\cap S^1\neq \{a,b\}$. This implies since $\tilde{V}_0$ is a connected component of the complement of $p^{-1}(\mathcal{M}^c)$ that $\tilde{\xi}\subset \tilde{V}_0$ which is impossible by definition of $\tilde{V}_0$.

Therefore we have that $A$ is an annulus verifying the points 1 and 2 of the Theorem. The points 3 and 4 are direct consequence of the construction.

\end{proof}

\begin{lema} Let $U \in \mathcal{A}\cap \Pi_0({\cal M}^c)$, and  ${\cal F}$ as in the previous lemma. Then, ${\cal F} = \{U\}$.
\end{lema}

\begin{proof} Suppose for a contradiction that ${\cal F}\neq \{U\}$, and let $A\in {\cal A}$ be given by the previous lemma. By item \ref{tres} in the previous lemma, $\partial A\subset {\cal M}$.  Moreover, ${\cal M}\cap A\neq \emptyset$ because ${\cal F}\neq \{U\}$. Therefore there exists $x\in \partial A$ and $n> 0$ such that $f^n(x)\in A$.  We claim that $f^n(U)_{\infty}= A_{\infty}$. First note that by item \ref{cuatro} in the previous lemma there exists $U_0\in {\cal F}$ and $z\in\partial U_0$ such that $f^n(z)\in A$. Then, $f^n(U_0)\cap A \neq \emptyset$, $f^n(U_0)\cap {\cal M} = \emptyset$ and so $f^n(U_0)\subset A$.  This implies that $f^n(U)_{\infty}= f^n(U_0)_{\infty}= A_{\infty}$.

We claim that $f^n(\partial A) = \partial A$.  Otherwise, because of minimality of $\mathcal{M}$ we have that $A\setminus f^n(A)$ is non-empty. Furthermore, due to the point 4 of the last lemma there exists an annulus $A_0\in\Pi_0(\mathcal{M}^c)$ such that $A_0\subset A\setminus f^n(A)$.

Let us fix $\{a,b\}=S^1\cap \tilde{\alpha}$ where $\tilde{\alpha}$ is the lift of an essential loop in $U_0$  and a lift $g$ of $f^n$ such that $g$ leaves invariant $\{a,b\} $. Further let $\tilde{A}$ be the lift of $A$ which contains $\tilde{\alpha}$ and $\tilde A_0\subset \tilde A$ lift $A_0$. So, $\tilde A_0\subset \tilde A\backslash \tilde g(\tilde A)$ which implies $\tilde g^{-1}(\tilde A_0)\notin \tilde A$.  But $f^{-n}(A_0)\in {\cal F}$, so $g^{-1}(\tilde A_0)\subset \tilde A$. This contradiction proves the claim.

However, $f^n(\partial A) = \partial A$ contradicts the fact that $x\in \partial A$ and $f^n(x)\in A$.

\end{proof}

\begin{cor}\label{anu} The number of essential annuli in $\Pi_0 ({\cal M}^c)$ is finite.
\end{cor}

\begin{proof}
Let $\mathcal{A}\subset \Pi_0(\mathcal{M}^c)$ be the family of all the essential annuli. Hence to each element $U\in\mathcal{A}$ we can associate a simple and closed geodesic $\gamma_U$. Furthermore, due to the previous lemma, the set $\mathcal{G}=\{\gamma_U:U\in\mathcal{A}\}$ is a pairwise disjoint family. Therefore, since the cardinal of any family of pairwise disjoint simple and closed geodesics is upper bounded by $3g-3$ (see \cite{bound}), we have that $\mathcal{G}$ is finite. Thus $\mathcal{A}$ is finite and its cardinal is upper bounded by $3g-3$.
\end{proof}

\begin{cor}\label{fini} The number of essential  components in $\Pi_0 ({\cal M}^c)$ is finite.
\end{cor}

\begin{proof}  It is obvious that the number of components with genus in  $\Pi_0 ({\cal M}^c)$ is finite.  Furthermore, as $i_{*}: \pi_1(C,x_0)\to \pi_1(S, x_0)$ is injective for $C\in \Pi_0 ({\cal M}^c)$ (see Lemma \ref{inj}), the number of genus zero components which are not annulus in  $\Pi_0 ({\cal M}^c)$ is also finite.  We are now done by Corollary \ref{anu}.
\end{proof}

\begin{cor}\label{peri}  Every essential component $U$ in $\Pi_0 ({\cal M}^c)$ is periodic. In particular, ${\cal M} = \partial U \cup \partial f(U) \cup \ldots\cup \partial f^p (U)$.
\end{cor}

We finish this section with the proof of Proposition \ref{cont}

\begin{proof} By definition of type 2 minimal set, there exists an essential component $U$ in $\Pi_0 ({\cal M}^c)$. The previous Corollary implies that ${\cal M} = \partial U \cup \partial f(U) \cup \ldots\cup \partial f^p (U)$. Furthermore, Corollary \ref{finiteboundcomp} implies that 
\\$\mathcal{F}=\{C\subset S: C\in\Pi_0(\partial f^i(U))\mbox{ for some }i=0,...,p\}$ is a finite family. Hence if we take ${\cal C_0}$ as to be a connected component of ${\cal M}$ contained $\partial U$, it has to be periodic. Thus due to the minimality of $\mathcal{M}$ we are done.
\end{proof}

\end{subsection}

\begin{subsection}{Surgery}

If $K\subset S$ is a continuum and ${\cal B} = \{\beta _1, \ldots, \beta _n\}$ is a finite family of loops in $S\backslash K$, we define $S^*_{\cal B}$ to be the closure of the connected component of $S\backslash \cup_{i=1}^n \beta _i $ containing  $K$.

\begin{figure}[ht]

\centerline{\subfigure[]{\psfrag{Beta}{$\beta$}\includegraphics[scale=0.35]{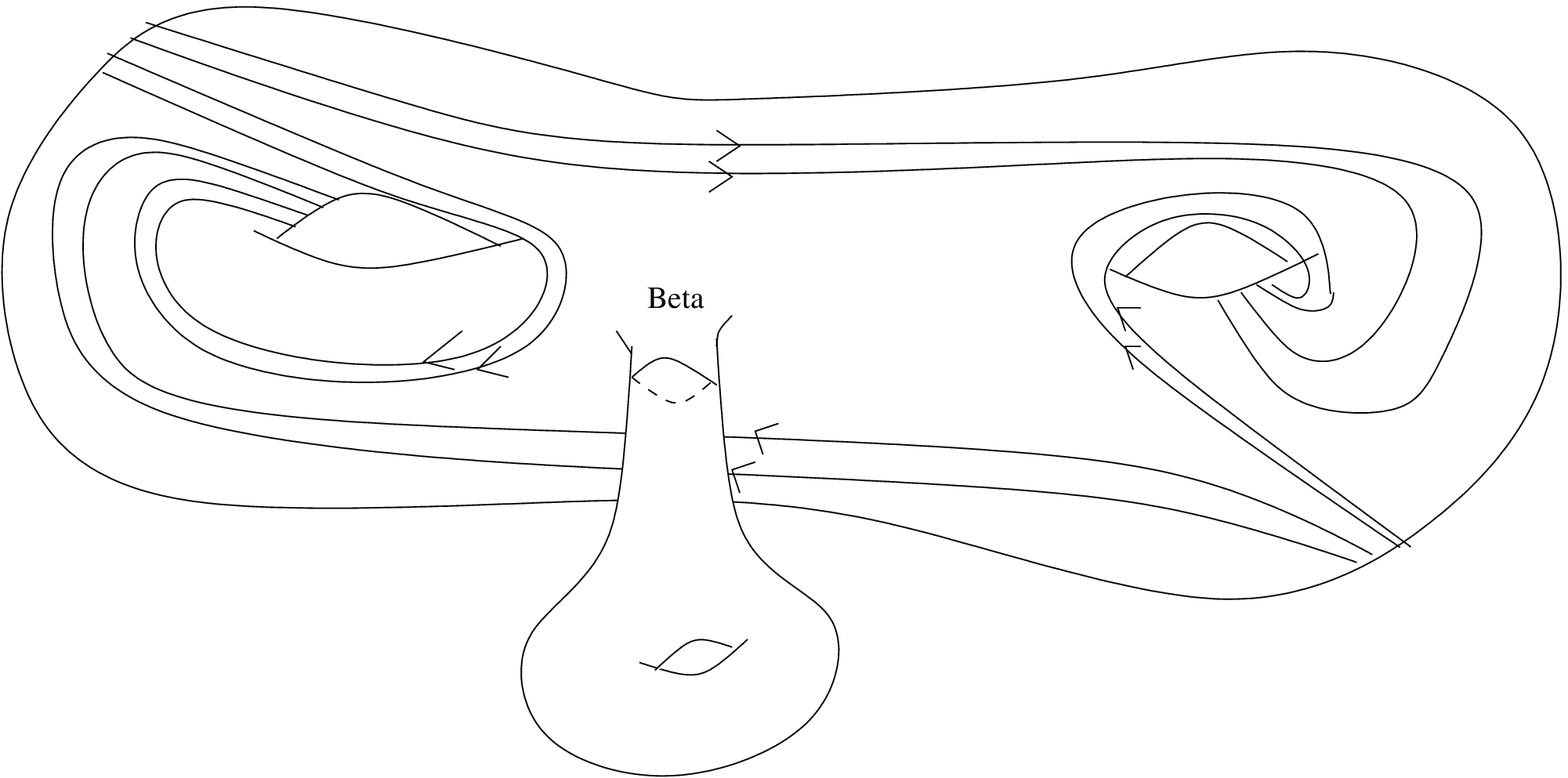}}\hspace{.1in}
\subfigure[]{\includegraphics[scale=0.35]{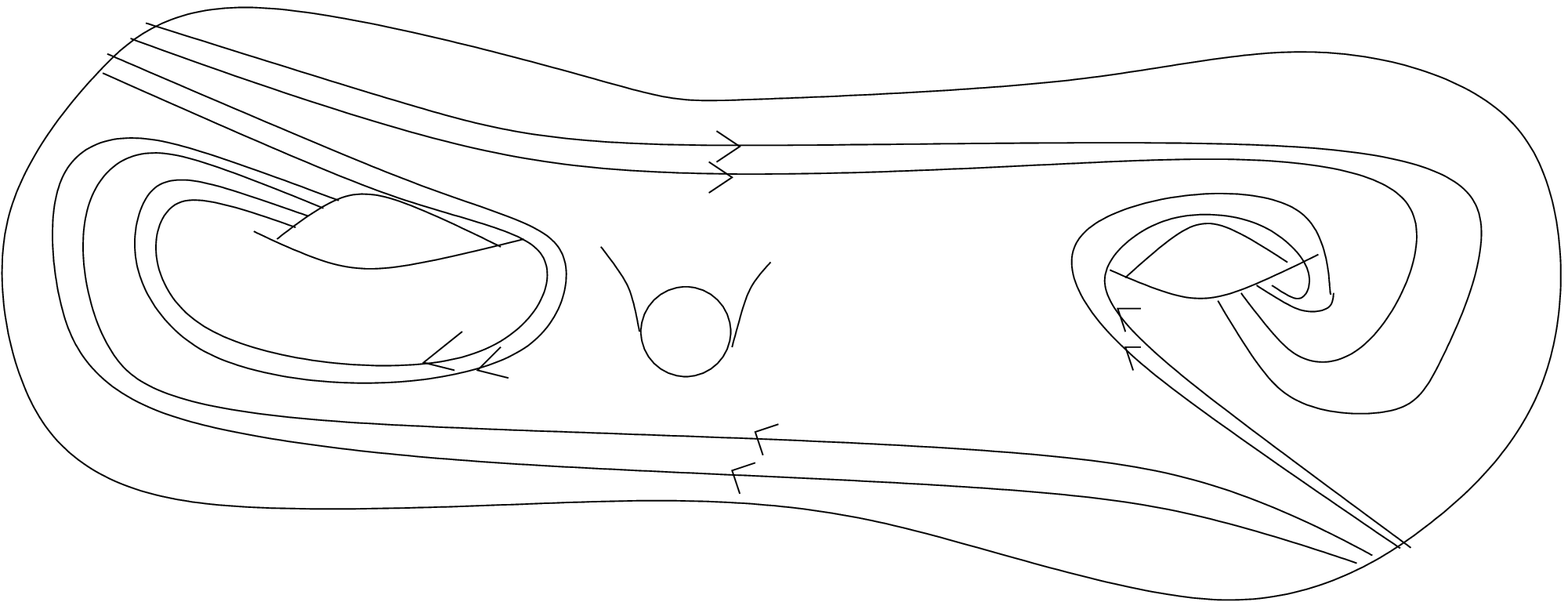}}}

\caption{${\cal B}$ and $S^*_{\cal B}$}\label{surgery}
\end{figure}

\begin{lema}\label{surg} Let $K\subset S$ be a continuum such that:

\begin{itemize}
\item $\#\{{\cal E}(S\backslash K)\} <\infty$
\item If $C\in  {\cal E}(S\backslash K)$ , then $i_*: \pi_1 (C, x_0)\to \pi_1 (S,x_0)$ is injective.
\end{itemize}

Then, there exists a finite family of simple, essential and pairwise disjoint loops in $S\backslash K$, ${\cal B} = \{\beta _1, \ldots, \beta _n\}$  such that:

\begin{enumerate}

\item \label{ii}the elements in $\Pi_0 (S^*_{\cal B}\backslash K)$ are either annuli with one boundary component contained in $\cup_{i=1}^n \beta _i$ or disks with a disjoint closure of $\bigcup_{i=1}^n\beta_i$;

    \item \label{iii}If  $g: C\to C$ is a homeomorphism, where  $C\in {\cal E}(S\backslash K)$ , then $\{[\beta]: \beta\in {\cal B}\cap C\} = \{[g(\beta)]: \beta\in {\cal B}\cap C\}$, where $[\gamma]\in \pi_1(S)/_{\alpha \sim \alpha ^{-1}}$
        \end{enumerate}
\end{lema}

\begin{proof} If $C\in {\cal E}(S\backslash K)$, the fact that $i_*: \pi_1 (C, x_0)\to \pi_1 (S,x_0)$ is injective implies due to Corollary \ref{puncturedsurf} that $C$ is homeomorphic to a surface of finite type (i.e. finite genus and finitely many punctures). That is, $C$ is homeomorphic to a $k_C$-punctured surface, $k_C\in \N, k_C\geq 1$. Pick a simple loop $\beta^C _i, i=1, \ldots, k_C$ surrounding each puncture in such a way that the $\beta^C _i$'s are pairwise disjoint, and define ${\cal B}_C= \{\beta ^C_1, \ldots, \beta^C _{k_C}\}$.  By the choice of the $\beta^C_i$'s, for any $i=1,\ldots,k_C$, $C\backslash \beta^C_i$ has two connected components, one of which is an annulus $A^C_i$ with one boundary component $\beta^C _i$ and the other contained in $K$.
Define $${\cal B} = \bigcup _{C\in {\cal E}(S\backslash K)} {\cal B}_C.$$
If $U\in{\cal E}(S^*_{\cal B}\backslash K)$, then $U = A^C_i $ for some $C\in \Pi_0 (S\backslash K)$ and $i=1,\ldots,k_C$. On the other hand if $U\in \Pi_0(S^*_{\mathcal{B}}\setminus K)\setminus {\cal E}(S^*_{\cal B}\backslash K)$ we have that it is trivial and, due to $K$ be connected, simply connected. This gives us \ref{ii}.

Furthermore, by construction, $\{[\beta]: \beta \in {\cal B}\cap C\}$ is determined by the topology of $C$, which gives \ref{iii}.
\end{proof}

Notice that in the last lemma we have by construction for every loop in $\partial S^*_{\beta}$, that they are accumulated by points in $S^*_{\beta}$ only by one side, with respect some fixed orientation. In particular, this implies that $S':=S^*_{\beta}/_{\partial S^*_{\beta}}$ is a closed surface. Furthermore, again by construction, we have that $g(S')<g(S)$.

If ${\cal C}\subset S$ is a continuum and type 2 minimal set for $f\in \homeo$, then ${\cal C}$ verifies the hypothesis of Lemma \ref{surg} (see Lemma \ref{inj} and Corollary \ref{fini} and note that the existence of an essential component $U\in \Pi_0 ({\cal M}^c)$ forbids the existence of a fully essential component in $\Pi_0 ({\cal M}^c)$).

\begin{lema}\label{no}  Let ${\cal C}$ be a continuum and type 2 minimal set for $f\in \homeo$, and let ${\cal B}$ be given by Lemma \ref{surg}.  Then, there exists $n_0\in \N$ and $g:S\to S$ such that $g=f^{n_0}$ in ${\cal M}$ and $g(\beta) = \beta$ for all $\beta \in {\cal B}$.

\end{lema}

\begin{proof}  By Corollary \ref{peri}, essential components in $\Pi_0({\cal C}^c)$ are periodic.  So, there exists $n\in \N$ such that $f^{n} (C)=C$ for all essential $C\in \Pi_0({\cal C}^c)$.  Furthermore, by Lemma \ref{surg}, Item \ref{iii} $\{[\beta]: \beta\in {\cal B}\cap C\} = \{[f^n(\beta)]: \beta\in {\cal B}\cap C\}$.  As ${\cal B}$ is a  finite family, we take a power $n_0$ of $f$ fixing every $[\beta]$, $\beta\in {\cal B}$.  It is a well known fact that one can modify $f^{n_0}$ with an isotopy supported in ${\cal M}^c$ and obtain $g:S\to S$ such that $g(\beta) = \beta$ for all $\beta \in {\cal B}$.
\end{proof}

Let $\cont\subset S$ be a continuum and type 2 minimal set of $f\in \homeo$. Further, consider the family of pairwise essential loops ${\cal B}$ given by Lemma \ref{surg} and $S^*=S^*_{\cal B} $. Let $n_0$ and $g$ be given by Lemma \ref{no}. We give a proof for Theorem \ref{typ2str2}:

\begin{proof}  The quotient map is a homeomorphism restricted to $\inte S^*$, and therefore conjugates $g|_{\cal C}$ and $\hat g|_{\cal C'}$, which gives us $\ref{pri}$. By Lemma \ref{surg}, item \ref{iii}, $S^*\backslash {\cal C}$ is a union of disks or annuli with one boundary component contained in ${\cal B}$. So, $S'\backslash {\cal C'}$ is a union of disks. Furthermore, for each $C\in \mathcal{E}(\cont^c)$ there is at least one $\beta \in {\cal B}\cap C$, which gives us \ref{segu}.

\end{proof}

\begin{figure}[ht]

\centerline{\subfigure[]{\includegraphics[scale=0.5]{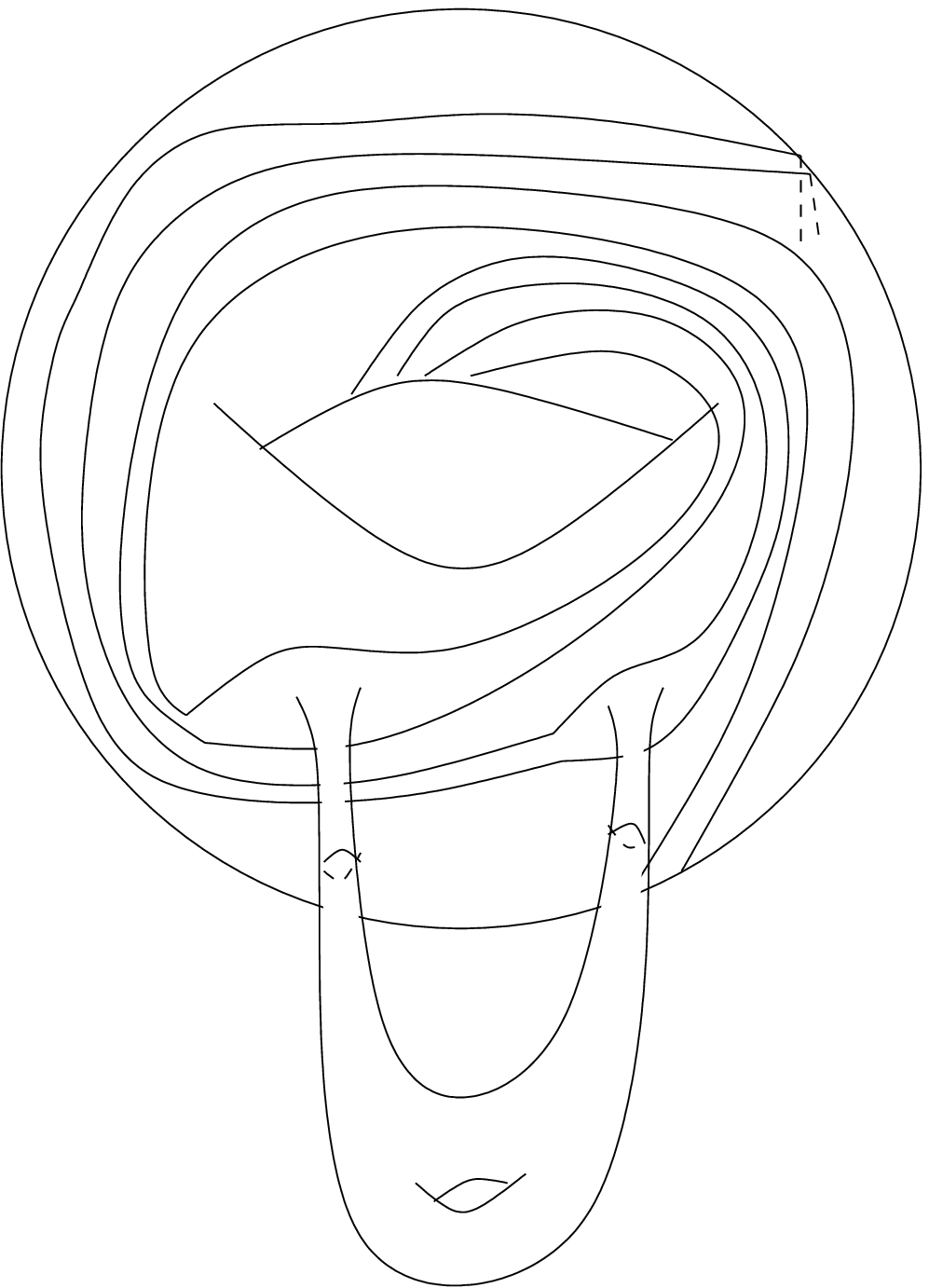}}\hspace{.2in}
\subfigure[]{\includegraphics[scale=0.5]{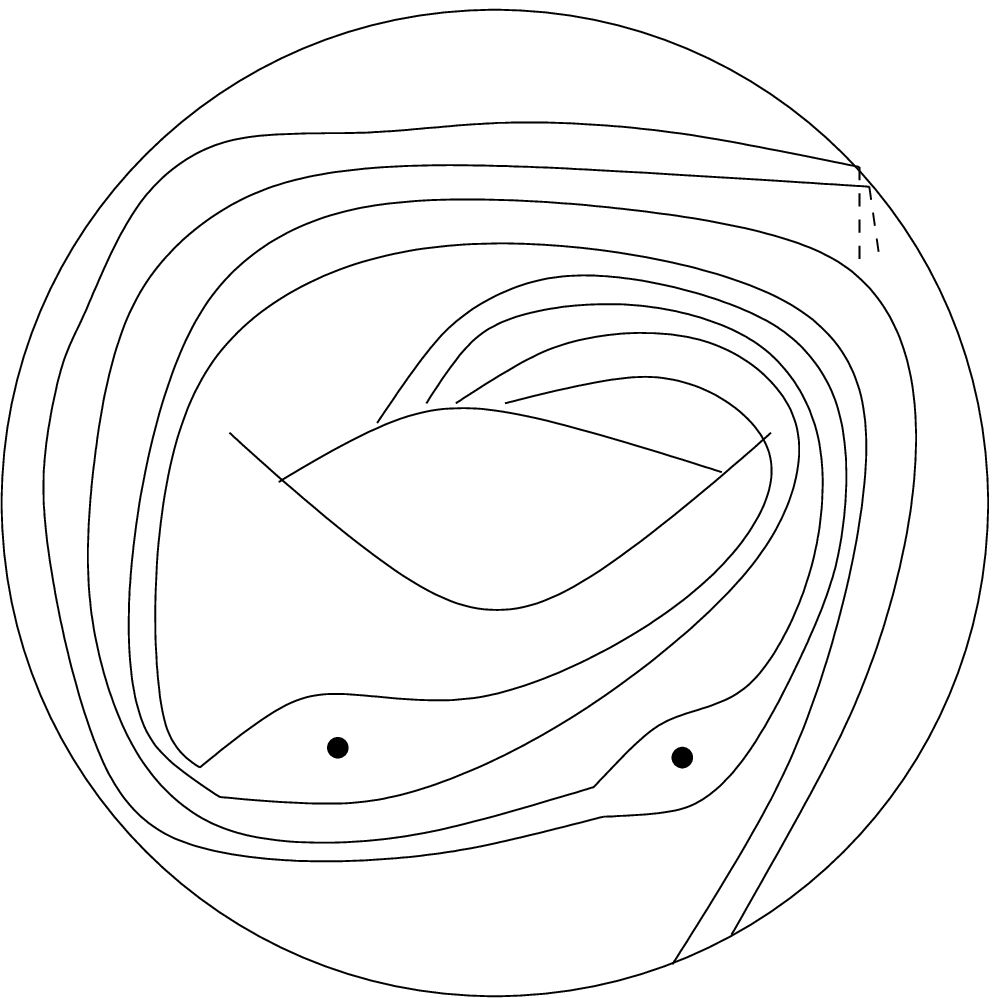}}}

\caption{Proof of Theorem \ref{typ2str2} }\label{42}
\end{figure}

\end{subsection}

\end{section}

\begin{section}{Type 3 minimal sets.}\label{type3}

As usual, we consider a closed hyperbolic surface $S$.  Recall that a minimal set $\mathcal{M}\subset S$ of $f\in\homeo$ is an extension of a minimal Cantor set (of a periodic orbit) if there exists a map $h:S\rightarrow S$ homotopic to the identity such that:

\begin{itemize}

\item[(i)] $h\circ f=\hat{f}\circ h$ for some $\hat{f}\in\homeo$;

\item[(ii)] $\mathcal{M}'=h(\mathcal{M})$ is a minimal cantor set (is a periodic orbit) of $\hat{f}$, and $h^{-1}(y)$ contains a unique connected component of $\mathcal{M}$ for every $y\in\mathcal{M}'$.

\end{itemize}

In this section we prove the following

\begin{prop}\label{3} If ${\cal M}$ is a type 3 minimal set, then $\mathcal{M}$ is  an extension of a periodic orbit or a minimal cantor set.
\end{prop}

The analogue result for the torus is proved in $\cite{Kwajapa}$. As much of the work is already done in that paper, rather than repeating the proofs we give precise references.

The topological tool needed for the proof is a theorem by Roberts and Steenrod \cite{rs} generalizing for surfaces a result of Moore for the sphere \cite{moore}.  It concerns upper semi-continuous partitions of $S$ into continua.

A partition $G$ of $S$ into closed sets is upper semi-continuous if for each $g\in G$ and each open set $U$ containing $g$, there exists an open set $V\subset U$ containing $g$ and such that every $g'\in G$ which meets $V$ lies in $U$.

In our context ($S$ is a compact metric space), $G$ is upper semi-continuous if and only if for every $\{g_n\}_{n\in\N}\subset G$ with a Hausdorff limit $X$, there exists $g\in G$ such that $X\subset g$.

We say that a set $X\subset S$ is {\it filled contractible} if in each open neighborhood of $X$ contains a disk $D$ having $X$ in its interior, such that $D\setminus X$ has a unique connected component.

\begin{teo}\label{mo}\emph {\cite {rs}} If $G$ is an upper semi-continuous partition of $S$ such that every $g\in G$ is a filled contractible continuum, then $S/G$ is homeomorphic to $S$.
\end{teo}

If  a closed set $X\subset S$ is contained in a bounded disk $D\subset S$, then $\lleno (X)$ is defined as to be the union of $X$ with the bounded components of its complement.  So, $\lleno (X)$  is a filled contractible set.

As ${\cal M}$ is a type 3 minimal set, then every $\cont \in \Pi_0({\cal M})$ is contained in a bounded disk (see Lemma \ref{ltop}).  So, $\lleno(\cont)$ is defined for all $ \cont \in \Pi_0({\cal M})$.

\begin{lema}\label{ups}  Let $G=\{\lleno({\cal C}): {\cal C}\in \Pi_0({\cal M})\} \cup \{x\in S : x\notin \cup _{{\cal C}\in \Pi_0({\cal M}}\lleno({\cal C}) \}$. Then, $G$ is an upper semi-continuous partition of $S$ in filled contractible sets.
\end{lema}

\begin{proof} We first prove that $G$ is a partition. It is enough to prove that $\lleno({\cal C})\cap \lleno({\cal C}')=\emptyset$ if ${\cal C}$ and ${\cal C}'$ are different elements in $\Pi_0({\cal M})$. But $\lleno({\cal C})\cap \lleno({\cal C}')\neq \emptyset$ implies that $\lleno({\cal C})\subset \lleno({\cal C}')$ (or the other way round). However, this is impossible because of Lemma \ref{bounded}.

Let $\{g_n\}_{n\in\N}\subset G$ with a Hausdorff limit $X$. As every set in $g$ is a continuum, then $X$ is a continuum. Furthermore, by definition of $\lleno({\cal C})$ we have for every point $x\in X$ a sequence $(y_n)_{n\in\N}$ such that $y_n\in g_n\cap \mathcal{M}$ for every $n\in\N$ and $(y_n)_{n\in\N}$ converges to $x$ (see \cite{Kwajapa} for a detailed argument). Besides, as ${\cal M}$ is closed, $X\subset {\cal M}$. So, there exists $g\in G$ such that $X\subset G$.

\end{proof}

The proof of Proposition \ref{3} follows exactly as in \cite{Kwajapa} (see Lemma 3.18, and the proof of Addendum 3.17)

\end{section}

\begin{section}{Proof of Theorem A and Theorem B.}\label{proofAB}

In this section we prove the main results of the article, which are a classification of minimal set for surface homeomorphisms given by Theorem A, and the topological and dynamical characterization for type 2 minimal sets given in Theorem B (see the introduction).

\begin{teo}\label{teoA}

Let $S$ be a closed hyperbolic surface and $\mathcal{M}\subset S$ minimal set of $f\in\homeo$. Then either:

\begin{itemize}

\item[(1)] the elements in $\Pi_0(\mathcal{M}^c)$ are all disks, with at least one unbounded. Moreover, any bounded disk in $\Pi_0(\mathcal{M}^c)$ is wandering;

\item[(2)] $\mathcal{M}=\cont_0\cup...\cup\cont_{n-1}$, where $\{\cont_i\}_{i=0}^{n-1}$ is a family of pairwise disjoint continua each containing a connected component of $\partial V$ for some $V\in \mathcal{E}(\mathcal{M}^c)$, and  $f(\cont_i)=\cont_{i+1}$ $\textrm{mod}(n)$ for all $i\in \{0, \ldots, n-1\}$.

\item[(3)] $\mathcal{M}$ is either an extension of a periodic orbit or an extension of a minimal cantor set.

\end{itemize}
\end{teo}

\begin{proof}

Theorem \ref{prel} gives us three complementary cases:
\begin{enumerate}

\item Every $U\in\Pi_0(\mathcal{M}^c)$ is a disk;

\item There exists an essential component $U\in\Pi_0(\mathcal{M}^c)$, and 
\\$i_{*}:\pi_1(V,x_0)\rightarrow \pi_1(S,x_0)$ is injective for all $V\in\Pi_0(\mathcal{M}^c)$;

\item There exists a fully essential component $V\in\Pi_0(\mathcal{M}^c)$ and every other component $U\in\Pi_0(\mathcal{M}^c)$ is a bounded disk.

\end{enumerate}

In the first case, Lemma \ref{36} gives us an unbounded periodic disk in $\mathcal{M}^c$, and Lemma \ref{35} gives us that any bounded disk in $\Pi_0(\mathcal{M}^c)$ is wandering.

In the second case, Proposition \ref{cont} gives us the desired result.

Finally, in the third case we apply directly Theorem \ref{3}.

\end{proof}

From now on we refer to a minimal set $\mathcal{M}$ as a type 1,2 or 3 minimal set, whenever it is in the case 1,2 or 3 of the last theorem.

In what follows we define some special continua which arise naturally as minimal sets for surface homeomorphisms, as one can see in the toral case (see \cite{Kwajapa}). Let $\cont\subset S$ be a continuum, we say that it is an essential \textit{circloid} if it verifies:

\begin{itemize}

\item[(i)] there exists some essential annuli $U$ such that $\cont\subset U$, and $U\setminus \cont$ has two unbounded components (components with non trivial fundamental group);

\item[(ii)] The continuum $\cont$ is minimal with respect the condition (i), that is, it does not strictly contain any other continuum satisfying property (i).

\end{itemize}

When $\cont\subset S$ only verifies (i), we say that $\cont$ is an annular continuum. The topology of such a continuum can be as simple as a loop or as complicated as a pseudo-circle. The following result can be find in \cite{Kwajapa}:

\begin{lema}\label{circmin}

If $\cont\subset S$ is an annular continuum which is a minimal set for $f\in\homeo$, then $\cont$ is a the boundary of a circloid.

\end{lema}

Before stating the main theorem of this section, we need a result which is contained in the S. Matsumoto's and H. Nakayama's work \cite{MaNa}, concerning continua which are minimal set for homeomorphisms in the sphere.

\begin{teo}\label{ponjas}

Let $\cont\subset S^2$ be a continuum and minimal set for a homeomorphism $f\in\textrm{Homeo}(S^2)$. Then, there are at most two components in $\Pi_0(S^2\setminus \cont)$ which are periodic by $f$.

\end{teo}

Now we are ready to give a proof for the Theorem B.

\begin{teo2}\label{teoB}

Let $S$ be a closed hyperbolic surface and $\mathcal{M}=\cont_0\cup...\cup\cont_{n-1}$ a type 2 minimal set of $f\in\homeo$. Then for every $k=0,...,n-1$ there exists $N>0$, such that $\cont_k$ is reducible to a type 1 minimal set for $f^N$. Furthermore,

\begin{itemize}

\item[(2i)] if $g(S') = 0$, then $\cont_k$ is the boundary of a circloid for every $k=0...n-1$ and $\mathcal{E}(\cont_k^c)\leq 2$. Further, every disk in $\Pi_0(\mathcal{M}^c)$ is wandering.

\item[(2ii)] if $g(S') > 0$ every bounded disk in $\mathcal{M}^c$ is wandering, and there are at least $\#\{{\cal E}({\cal C}_k^c)\}$ unbounded fixed disks in $\Pi_0 (S'\backslash{\cal C}_k')$.

\end{itemize}

\end{teo2}

\begin{proof}

For a fixed $k\in\{0,...,n-1\}$ we consider $S^*\supset \cont_k$ and $g\in \homeo$ associated to $\cont_k$ and $f^{n+n_0}$ as in Theorem \ref{typ2str2}. Let $N=n+n_0$. Then due to Theorem \ref{typ2str2} we have that $\cont_k$ is reducible to a type 1 minimal set for $f^N$.

Furthermore, if we consider $S'$ given by Theorem \ref{typ2str2} we have the two following complementary cases:

\begin{itemize}

\item[(2i)] $g(S')=0$: in this case we have that $\cont'$ is a minimal set of $\hat{g}$ in $S'$, and $S'$ can be identify with the sphere $S^2$ . Furthermore $\cont'$ is a continuum contained in $q(\inte S^*)$ and necessarily there are at least to different essential loops $\beta_1,\beta_2$ in $\partial S^*$ (otherwise $S'$ must have genus bigger than 0). Hence Theorem \ref{typ2str2} implies that there are at least two fixed disks $U_0,U_1\in\Pi_0(\cont'^c)$ of $\hat{g}$ separated by $\cont'$. Moreover due to Theorem \ref{ponjas} these two are the unique loops in $\partial S^*$.

Consider now a neighborhood $V\subset S^2\setminus \{q(\beta_1),q(\beta_2)\}$ of $\cont'$, which is an essential annulus in the open annulus $S^2\setminus \{q(\beta_1),q(\beta_2)\}$. Thus due to $q:\inte S^*\to S^2\setminus \{q(\beta_1),q(\beta_2)\}$ be a homeomorphism, we have that $W=q^{-1}(V)$ is an essential annulus containing $\cont$. This, together with Lemma \ref{circmin} implies that $\cont_k$ is the boundary of a  circloid. Hence the same holds for every $k=0,...,n-1$, so we are in case (2i) of the Theorem.

\item[(2ii)] $g(S')>0$: in this case the point (2ii) of the Theorem is a direct consequence of Theorem \ref{typ2str2} and Lemma \ref{35}.

\end{itemize}

\end{proof}

\end{section}

\begin{section}{Classification Theorem in the non-wandering case.}\label{nonwand}

We now restrict our classification to the class of non-wandering homeomorphisms. Recall that $f:S\rightarrow S$ is non-wandering if for every neighborhood $U$ in $S$ there exists $n\in\N$ such that $f^n(U)\cap U\neq\emptyset$. We denote by $\textrm{Homeo}_{nw}(S)$ the class of all such homeomorphisms.

As was done in the toral case(\cite{Kwajapa}), we make use of a result by A. Koropeki (\cite{koro}) which describes the topology of continua in surfaces which are $f$-invariant and aperiodic (i.e it does not contain periodic orbits) for some non-wandering homeomorphism $f$.

\begin{teo}

Suppose that $\cont\subset S$ is an $f$-invariant and aperiodic continuum for $f\in\textrm{Homeo}_{nw}(S)$. Then $\cont$ is an intersection of a decreasing sequence of annuli.

\end{teo}

As a consequence, our classification theorem in the non-wandering setting comes down to:

\begin{cor}\label{classnw}

Let $S$ be an oriented surface and $\mathcal{M}\subset S$ a minimal set for $f\in\textrm{Homeo}_{nw}(S)$. Then, either:

\begin{itemize}

\item[$(1^{nw})$] $\mathcal{M}=\cont_0\cup...\cup\cont_{n-1}$ where $\{\cont_i\}_{i=0}^{n-1}$ is a family of pairwise disjoint circloids each containing a connected component of $\partial V$ with $V\in \mathcal{E}(\mathcal{M}^c)$, and $f(\cont_i)=\cont_{i+1}$ $\textrm{mod}(n)$. Further there are no disks in $\Pi_0(\mathcal{M}^c)$;

\item[$(2^{nw})$] $\mathcal{M}$ is an extension of either a periodic orbit or a minimal cantor set.

\end{itemize}

\end{cor}

\begin{proof}

We first want to see that type 1 and type 2ii minimals sets can not occur for $f\in\textrm{Homeo}_{nw}(S)$. If $S'$ is an oriented surface and $\mathcal{M}\subset S'$ is a type 1 minimal set for some homeomorphism $g\in\textrm{Homeo}(S')$, then $\mathcal{M}$ is an aperiodic continuum which is not a decreasing intersection of annuli. Hence, due to the result above, type 1 minimal sets can not occur for $g\in\textrm{Homeo}_{nw}(S')$.

This implies that type 1 minimal sets can not occur for $f\in\textrm{Homeo}_{nw}(S)$. Suppose for a contradiction that
\\$\mathcal{M}=\cont_0\cup...\cup \cont_{n-1}$ is a minimal set of $f\in\textrm{Homeo}_{nw}(S)$ in the case 2ii of Theorem B. The map $\hat{g}$ associated by Theorem B to $\cont_0$, can be consider by construction non-wandering (we do not give a proof, but this fact is well known). On the other hand $\cont_0'$ is a type 1 minimal set for $\hat{g}$, which is absurd.

So, now we consider $\mathcal{M}$ as in 2i of Theorem B. We know that
\\$\mathcal{M}=\partial\cont_0\cup...\cup \partial\cont_{n-1}$ where $\{\cont_i\}_{i=0}^{N-1}$ is a family of pairwise disjoint circloids. Further, by definition of circloid we know that the existence of an interior point $x$ for some $\cont_i$ would imply the existence of a disks in $\Pi_0(\mathcal{M}^c)$, which has to be wandering. Hence due to $f$ be non-wandering we have  that $\partial\cont_i=\cont_i$ for every $i=0...N-1$, which finish the proof.

\end{proof}

We point out the slight difference between the cases $(1^{nw})$ of last theorem and $(2i)$ of Theorem B. In the non-wandering case we have that the continua are circloids, while in the general case we only know that they are the boundary of a circloid.

\end{section}

\begin{section}{Classification and Nielsen-Thurston theory.}\label{nt}

In this section we study the given classification of minimal sets in light of Nielsen-Thurston theory (see \cite{Thurston}). We will see that type 1 minimal sets can only occur in the homotopy class of either periodic homeomorphisms or reducible maps without pseudo-Anosov components. This will trivially imply for elements in the homotopy class of pseudo-Anosov homeomorphisms, that only type 3 minimal sets can occur (which generalizes the analogous result given in \cite{Kwajapa}). Let us first introduce some classical notions of this theory (see \cite{Boyland1}).

Consider a metric $d$ in a space $\mathcal{X}$ and $f:\mathcal{X}\rightarrow\mathcal{X}$ bijective. We say that $f$ is expansive with respect $d$ if there exists $\alpha>0$ such that
\\$\sup\{d(f^n(x),f^n(y)):n\in\Z\}>\alpha$ for any $x,y\in\mathcal{X}$. Given now $f:\mathcal{X}\rightarrow\mathcal{X}$ and $g:\mathcal{X}\rightarrow\mathcal{X}$ bijective functions we say that the pairs $(x,f)$ and $(y,g)$ globally  shadows with respect $d$, if there exists $K>0$ such that $d(f^n(x),g^n(y))<K$ for every $n\in\Z$. Is easy to see that this define an equivalence relation between the pairs.

Consider a hyperbolic surface $S$ and two homeomorphisms $f,\phi\in\homeo$ which are homotopic . Further fix $\tilde{f},\tilde{\phi}$ equivariantly homotopic lifts of $f,\phi$. We say for $x,y\in S$ that $(x,\tilde{f})$ and $(y,\tilde{\phi})$ globally shadows if $(\tilde{x},\tilde{f})$ and $(\tilde{y},\tilde{\phi})$ globally shadows with respect the hyperbolic metric, where $p(\tilde{x})=x$ and $p(\tilde{y})=y$.

We denote by $\textrm{Homeo}_+(S)$ the set of orientation preserving homeomorphisms in $S$ and by $\rho$ the hyperbolic metric of $S$. In light of the Nielsen-Thurston theory we define:

$$\mathcal{A}=\{f\in\textrm{Homeo}_+(S)|\ f\sim g\ :\ g \mbox{ is either pA or reducible with pA components}\},$$

The following result can be found in \cite{Boyland2} (Theorem 3.2).

\begin{teo}\label{Boy}

Suppose $S$ hyperbolic surface and $f\in\mathcal{A}$. Then there exists $\phi\in\textrm{Homeo}_+(S)$ and a non-empty $\phi$-invariant and compact surface $S'\subset S$ (possibly with boundary and non-connected) such that:

\begin{itemize}

\item[(1)] the boundary of $S'$ is either empty or union of essential curves;

\item[(2)] $\phi$ is homotopic to $f$ and $\phi|_{\textrm{int}(S')}$ is expansive;

\item[(3)] There exist a compact $f$-invariant set $\mathcal{Y}\subset S$, a continuous homotopic to the projection map $\alpha:\mathcal{Y}\rightarrow S /_{\partial S'}$, and two equivariantly homotopic lifts $\tilde{f},\tilde{\phi}$ of $f,\phi$ verifying:

\begin{itemize}

\item[(3a)] $\textrm{Im}(\alpha)=S' /_{\partial S'}$ and $\alpha\circ f|_{\mathcal{Y}}=\phi_q\circ\alpha$, where $\phi_q:S'/_{\partial S'}\rightarrow S'/_{\partial S'}$ is the induced map by $\phi$;

\item[(3b)] if $x,x'\in\mathcal{Y}$ where $(x,\tilde{f})$ and $(x',\tilde{f})$ globally shadows, then $\alpha(x)=\alpha(x')$.

\end{itemize}

\end{itemize}

\end{teo}

The properties of the function $\alpha$ allow us to state the following Lemma.

\begin{lema}\label{uncollpase}

Consider $\alpha:\mathcal{Y}\rightarrow S/_{\partial S'}$ given by Theorem \ref{Boy} and assume that $\mathcal{M}\subset S$ is a type 1 minimal set of $f$. Then $\mathcal{Y}\subset \mathcal{M}^c$.

\end{lema}

\begin{proof}

Suppose for a contradiction that $\mathcal{Y}\cap\mathcal{M}\neq\emptyset$. Thus because of the minimality of $\mathcal{M}$ we have that $\mathcal{M}\subset \mathcal{Y}$. This implies, since $\alpha$ semiconjugates
$f|_{\mathcal{Y}}$ with $\phi_q$, that $\mathcal{M}'=\alpha(\mathcal{M})$ is a minimal set of $\phi_q$. Furthermore due to $\mathcal{M}$ be a type 1 minimal set of $f$ and point (1) in the last Theorem, we have necessarily that
$\mathcal{M}\cap\partial S'\neq\emptyset$. Thus if $q:S'\rightarrow S'/_{\partial S'}$ is the quotient map we have that $\mathcal{M}'\supset\{P_0\}$, where $\{P_0\}=q(\partial S')$. Therefore since $P_0$ is fixed by $\phi_q$ we have that $\mathcal{M}'=\{P_0\}$.

Now, by Theorem \ref{Boy} we have that $S/_{\partial S'}\setminus \{P_0\}$ is a union of open surfaces non trivial fundamental group (see \cite{Boyland2}), while, since $\mathcal{M}$ is a type 1 minimal set, $q(\mathcal{M})^c$ is a union of disks. This implies that $\alpha$ can not be homotopic to $q$, which is absurd.

\end{proof}

We now state the non-existence of type 1 minimal sets for elements in $\mathcal{A}$.

\begin{teo}\label{classnielthu}

If $\mathcal{M}\subset S$ is a type 1 minimal set of $f\in\homeo$, then $f$ is not in $\mathcal{A}$.

\end{teo}

\begin{proof}

Assume that $f$ preserves orientation. For the complementary case, one just apply the proof we do to $f^2$. We suppose for a contradiction that $f\in\mathcal{A}$. Let us consider $S'$, $\phi$ and $\alpha:\mathcal{Y}\rightarrow S/_{\partial S'}$ given by Theorem \ref{Boy}. Due to Lemma \ref{uncollpase} we have that $\mathcal{Y}\subset \mathcal{M}^c$. Then we have that $\mathcal{Y}\subset \mathcal{M}^c$, which implies the existence of a family of pairwise disjoint closed disks $\{D_i\}_{i=0}^{N-1}$ and of pairwise disjoint elements of $\Pi_0(\mathcal{M}^c)$, $\{U_i\}_{i=0}^{N-1}$, such that:

\begin{itemize}

\item[(i)] $D_i\subset U_i$ for every $i=0...N-1$;

\item[(ii)] $\mathcal{Y}\subset \bigcup_{i=0}^{n-1}D_i$;

\item[(iii)] $f(\mathcal{Y}_i)=\mathcal{Y}_{i+1}$ mod($N$), where $\mathcal{Y}_i=\mathcal{Y}\cap D_i$.

\end{itemize}

Let us now consider the two lifts $\tilde{f},\tilde{\phi}$ of $f,\phi$ given by Theorem \ref{Boy}. We claim that if $x,y\in \mathcal{Y}\cap \mathcal{Y}_i$ for some $i=0...N-1$ then $(x,\tilde{f})$ and $(y,\tilde{f})$ globally shadows.

In fact for $x,y$ consider two lifts $\tilde{x},\tilde{y}\in \tilde{D}_i$ where $\tilde{D}_i$ is a connected component of $p^{-1}(D_i)$, and let $R=\max\{\textrm{diam}(D_i):i=0...N-1\}$. Then we have that $\sup\{d(\tilde{f}^n(\tilde{x}),\tilde{f}(\tilde{y})):n\in\Z\}\leq R$. Otherwise by definition of $R$ and the point (iii) we have that $\tilde{f}^k(\tilde{D}_i)\cap \tilde{D}_k\neq\phi$ and $\tilde{f}^k(\tilde{D}_i)\cap \sigma(\tilde{D}_k)\neq\phi$, where $\tilde{D}_k$ is a connected component of $p^{-1}(D_{i+k\textrm{ mod}(N-1)})$ and $\sigma\in \Gamma\setminus\{e\}$. This implies that $p(\tilde{D}_k\cup \tilde{f}^k(\tilde{D}_i)\cup \sigma(\tilde{D}_k))$ is contained in a essential element of $\Pi_0(\mathcal{M}^c)$ which is impossible since $\mathcal{M}$ is a type 1 minimal set.

Thus, by the point (3b) in Theorem \ref{Boy} we have that $\textrm{Im}(\alpha)$ has at most $N$ elements, which is absurd since $\textrm{Im}(\alpha)=S'/_{\partial S'}$.
\end{proof}

\begin{cor}\label{corforpA}

Let $S$ be a closed and oriented hyperbolic surface and $f\in\homeo$ homotopic to a pseudo-Anosov homeomorphism. Then every minimal set of $\mathcal{M}$ is type 3.

\end{cor}

Before to give the proof, we want to say that if is considered a pseudo-Anosov homeomorphism $\phi\in\homeo$, due to the expansiveness of such a system one can apply a result due to R. Mañé (\cite{mane}) to proof that every minimal set is either a periodic orbit or a minimal cantor set.

\begin{proof}

Theorem \ref{classnielthu} asserts that type 1 minimal set can not occur. On the other hand, we have seen in section \ref{type2} that the existence of a type 2 minimal set of $f$
would imply that $f_{\#}^n([\beta])=[\beta]$ for some $n\in\N$ and $\beta\subset S$ essential loop, which can not happen since $f$ is in the homotopy class of a pseudo-Anosov
homeomorphism. Then, every minimal set of $f$ has to be type 3.

\end{proof}

\end{section}

\begin{section}{Examples.}\label{examples}

In this final section we state examples for each non-trivial case of the Theorem A and Theorem B.

\begin{subsection}{Type 1 minimal sets.}

We say that a flow $\phi:S\times \R\rightarrow S$ is irreducible if:

\begin{itemize}
\item[(i)] $\phi$ has a unique minimal set $\mathcal{M}\subset S$;
\item[(ii)] $\mathcal{M}$ intersect any essential loop in $S$.
\end{itemize}

It is known in general the existence of $t\in\R$ such that $\mathcal{M}$ is a minimal set of the time-$t$ map $f_t$ of the flow (see \cite{timet}). Due to property (ii) $\mathcal{M}\subset S$ is a type 1 minimal set of $f_t$. Notice that by construction the map $f$ is homotopic to the identity, which agrees with Theorem \ref{classnielthu}.

Irreducible flows in any closed surface are constructed in \cite{flows}. For the sake of completion we give a hint below as how to construct such examples.

One can always construct an oriented geodesic lamination $\Lambda$ in the surface $S$ of genus $g>1$ such that $S\backslash \Lambda$ is the union of two disks that lift to the universal covering space to ideal $2g$- gones with alternating orientations on its sides.  Indeed, a standard way of constructing pseudo-Anosov maps is by lifting an Anosov diffeomorphism on $\T^2$ to a branched covering $\pi: S\to \T^2$.  If $\pi$ has exactly two ramification points where $\pi$ is $g-1$, then the stable (or unstable) lamination of the corresponding pseudo-Anosov has this property.

So, one can complete the lamination to a flow in $S$ with a fixed point of index $1-g$ in each disk.
Moreover, the lamination is a minimal set for the flow, and most importantly, there exists $t\in \R$ such that $\Lambda$ is minimal for the time $t$ of this flow (\cite{timet}).

Notice that by construction the map $f$ is homotopic to the identity, which agrees with Theorem \ref{classnielthu}.

\end{subsection}

\begin{subsection}{Type 2 minimal sets.}

Examples in this category can be constructed defining a homeomorphism of an essential annulus $A\subset S$ and then extending it to a homeomorphism $f$ of the whole surface $S$.  A minimal set of the annulus homeomorphism will of course be a minimal set of $f$.  The simplest example in this spirit is to take an irrational rotation in $A$, obtaining simple closed curves as minimal sets.  Of course, much wilder minimal sets can be obtained with this method.  The following Theorem appears in \cite{notassylvain} but gathers together previous ideas of Handel \cite{handel}, Herman \cite{herman} and Fayad-Katok \cite{fk}.

\begin{teo} There exists a homeomorphism of the annulus $S^1\times \R$ which has a minimal invariant set $\Lambda$ which is connected, not locally connected at any point and has positive Lebesgue measure.
\end{teo}

If one construct examples in this way, one obtains type 2 connected minimal sets ${\cal M}$.  Indeed, there is an essential component $U\in \Pi_0({\cal M}^c)$ which is a surface of genus $g-1$.  One may of course generalize this example by allowing the minimal set to have more than one connected component, and therefore different homotopy types for the essential components in $\Pi_0({\cal M}^c)$.  Take a symmetry $\varphi:S\to S$ taking an essential closed annulus $A\subset S$  to an essential closed annulus $ A'\subset S$, $A\cap A'=\emptyset$.  Let $f\in \homeo$ supported in a neighborhood $U$ of $A$, leaving $A$ invariant and such that the restriction of $f$ to $A$ has a minimal set ${\cal C}$.  Let $g= \varphi\circ f\circ \varphi ^{-1}(x)$. Then, $g$ is supported in $\varphi(U)$, leaves $A'$ invariant, and ${\cal C}'=\varphi ({\cal C})$ is a minimal set for $g|_{A'}$.  Define $h\in\homeo$ by $h=f\circ g$.  Now, ${\cal C}\cup {\cal C}'$ is a minimal set for the map $\varphi \circ h$.

\end{subsection}

\begin{subsection}{Type 3 minimal sets.}

Non-trivial examples of this type of minimal sets are those which are not periodic orbits nor Cantor sets. These examples can be easily constructed in such a way that the minimal set has finitely many connected components which are not a singleton but periodic. In \cite{Ferry} non-trivial examples with uncountable many connected components, some of them being not a singleton, are constructed.

\end{subsection}

\end{section}

\vspace*{0.2in}

\author{$\ $ \\

Alejandro Passeggi\\
 Emmy Noether Group: Low-dimensional and non-autonomous DynamicsEmmy,\\
 TU-Dresden,\\
 Dresden, Germany.\\
 \texttt{alepasseggi@gmail.com}}

\author{$\ $ \\

Juliana Xavier\\
 I.M.E.R.L,\\
 Facultad de Ingenier\'ia,\\
Universidad de la Rep\'ublica,\\
 Julio Herrera y Reissig,\\
 Montevideo, Uruguay.\\
 \texttt{jxavier@fing.edu.uy}}

\end{document}